%% file: main-revised-3.tex
\documentclass[12pt]{amsart}
\usepackage[left=2.75cm, right=2.75cm]{geometry}
\usepackage[utf8]{inputenc}
\usepackage{amsmath,mathtools,amsthm}
\usepackage{pinlabel,graphicx}
\usepackage{caption}
\usepackage{fancyhdr,lipsum}
\usepackage{quiver}

\usepackage{svg}
\usepackage{float}

\newtheorem{theorem}{Theorem}[section]

\newtheorem{lem}[theorem]{Lemma}

\newtheorem{defn}[theorem]{Definition}

\newtheorem{prop}[theorem]{Proposition}
\newtheorem{coro}[theorem]{Corollary}

\newcommand{\cp}{\mathbb{C}\mathrm{P}^1} 

\newcommand{\pslr}{\mathrm{PSL}_2(\mathbb{R})}
\newcommand{\pslc}{\mathrm{PSL}_2(\mathbb{C})}
\newcommand{\pslnc}{\mathrm{PSL}_n(\mathbb{C})}
\newcommand{\slnc}{\mathrm{SL}_n(\mathbb{C})}
\newcommand{\slc}{\mathrm{SL}_2(\mathbb{C})}

\newcommand{\sm}{\mathbb{S}, \mathbb{M}}
\newcommand*{\bigchi}{\mbox{\large$\chi$}}

\usepackage{svg}

\title{Meromorphic Projective Structures:\\ Signed Spaces, Grafting and Monodromy}

\date{}

\begin{document}

\author{Spandan Ghosh} 
\author{Subhojoy Gupta}

\address{The Mathematical Institute,
Radcliffe Observatory Quarter,
Oxford
OX2 6GG}
\email{spandan.ghosh@maths.ox.ac.uk}
\address{Department of Mathematics, Indian Institute of Science, Bangalore 560012, India}
\email{subhojoy@iisc.ac.in}

\begin{abstract}
    A meromorphic quadratic differential on a compact Riemann surface defines a complex projective structure away from the poles via the Schwarzian equation.  In this article we first prove the analogue of Thurston's Grafting Theorem for the space of such structures with signings at regular singularities. This extends previous work of Gupta-Mj which only considered irregular singularities. We also define a framed monodromy map from the signed space extending work of Allegretti-Bridgeland, and we characterize the $\pslc$-representations that arise as holonomy, generalizing results of Gupta-Mj and Faraco-Gupta. As an application of our Grafting Theorem, we also show that the monodromy map to the moduli space of framed representations (as introduced by Fock-Goncharov) is a local biholomorphism, proving a conjectured analogue of a result of Hejhal. 
\end{abstract}

\maketitle

\section{Introduction}

A marked and bordered surface is a pair $(\mathbb{S}, \mathbb{M})$ where $\mathbb{S}$ is a compact oriented surface of genus $g$ and $k$ boundary components, together with a non-empty set $\mathbb{M}$ of finitely many marked points, where each boundary component has at least one marked point. We shall assume that if $g=0$ then $\lvert \mathbb{M} \rvert \geq 3$.

Let  $\mathfrak{n} = (n_1,n_2,\ldots, n_k)$  be a tuple of positive integers, each $n_i\geq 3$, such that $(n_i-2)$ is the number of marked points on the $i$-th boundary. The set of marked points in the interior of the surface, which can be considered as punctures, is denoted by $\mathbb{P}$. This paper concerns the space of (signed) meromorphic projective structures $\mathcal{P}^\pm(\mathbb{S}, \mathbb{M})$ which is a $2^{\lvert \mathbb{P}\rvert}$-fold branched cover over the space  $\mathcal{P}(\mathbb{S}, \mathbb{M})$ (see section 2.3 for details). 

Recall that such a meromorphic projective structure is determined by a meromorphic quadratic differential $q$ on a compact Riemann surface  $X$ of genus $g$, with $k$ poles having orders $(n_1,n_2, \ldots, n_k)$ and $\lvert \mathbb{P} \rvert$ poles having order at most $2$, via the Schwarzian equation:
\begin{equation}\label{schw}
    u^{\prime\prime} + \frac{1}{2}q u = 0.
\end{equation}
Namely, the ratio of solutions of \eqref{schw} determines a \textit{developing map}  $f:\widetilde{X} \to \cp$ that is equivariant with respect to a \textit{holonomy/monodromy representation}  $\rho:\pi_1(X) \to \pslc$. Historically, these arose in the context of uniformizing a punctured sphere (see for example Chapters VIII and IX of \cite{dSG}). 

Note that we recover $(\sm)$ as the underlying topological surface via a real blow-up of the $k$ poles of orders at least $3$; the horizontal directions of $q$ at such a pole determine the marked points on the corresponding boundary circle. (Recall here that a pole of order $n$ has $n-2$ horizontal directions, i.e. tangent directions $\pm v$ where $q(v) \in \mathbb{R}^+$.)

\vspace{.05in}

Our first result in this paper is a more geometric parametrization of  $\mathcal{P}^\pm(\mathbb{S}, \mathbb{M})$ which is an analogue of Thurston's Grafting Theorem, which we now briefly recall.  For  a \textit{closed} oriented surface $S$ of genus $g\geq 2$, Thurston had introduced a ``grafting" deformation of any hyperbolic structure on $S$ that results in a new complex projective structure. This construction starts with a hyperbolic surface $X$ with  developing map $f:\widetilde{X} \to \mathbb{H} \subset \cp$, and a measured geodesic lamination $\lambda$  on $X$. The new complex projective structure is obtained by equivariantly inserting ``lunes" along the images of the leaves of the lift of $\lambda$ to the universal cover. Here a lune is a region of $\cp$ bounded by two circular arcs, and  its ``angle"  is determined by the transverse measure on $\lambda$; for details of this operation see \cite{Dum, Tani}. Thurston's Grafting Theorem (see \cite{KT92, BabExp}) asserts that the space $\mathcal{P}_g$ of complex projective structures on $S$ is \textit{parametrized} by such deformations, i.e.\ the \textit{grafting map} 
\begin{equation}\label{grmap}
    \text{Gr}: \mathcal{T}_g \times \mathcal{ML} \to \mathcal{P}_g
\end{equation}
is a homeomorphism. (Here $\mathcal{T}_g$ is the \textit{Teichm\"{u}ller space} of marked hyperbolic structures and $\mathcal{ML}$ is the space of measured laminations on $S$.) 
 
In our context, the analogues of the spaces in the left hand side of \eqref{grmap}, are  enhanced Teichm\"{u}ller  space $\mathcal{T}^\pm (\sm)$ and the space of signed measured laminations $\mathcal{ML}^\pm(\sm)$, where in both cases the signing is associated with the punctures in $\mathbb{P}$ (see \S2 for details). The former space already appears in previous literature in the context of marked and bordered surfaces (see, for example, \cite{All}). The latter is a space that we introduce, and should be related to the spaces of real $\mathcal{A}$- and $\mathcal{X}$-laminations that Fock-Goncharov introduced in \cite[\S12]{FG} (see also \cite{FG2}).

We shall prove: 
\begin{theorem}\label{thm1} 
    There is a grafting map 
    \begin{equation*}   \widehat{\text{Gr}}:\mathcal{T}^\pm(\sm) \times \mathcal{ML}^\pm(\sm)  \to \mathcal{P}^\pm(\sm)
    \end{equation*}
    which is a homeomorphism.
\end{theorem}

In the case when the set $\mathbb{P} =\emptyset$, this was proved in \cite{Gup-Mj}; the key technical step in the proof of Theorem \ref{thm1} is to determine  how the signings at the points of $\mathbb{P}$ play a role. In particular, we shall crucially use the relation between the signed grafting map and the Schwarzian derivative at these poles (see Lemma \ref{lem:regexp}). 

Along the way, we shall provide a ``grafting description"  of the projective structures corresponding to quadratic differentials with \textit{simple} poles -- see Corollary \ref{cor:simp} -- which could be of independent interest. 

\vspace{.05in}

Next, we introduce a monodromy map 
\begin{equation}\label{mmap} 
    \widehat{\Phi}: \mathcal{P}^\pm (\sm) \to \widehat{\bigchi}(\sm)
\end{equation}
where the target is the space of \textit{framed} $\pslc$-representations of the fundamental group $\pi_1(\mathbb{S} \setminus \mathbb{P})$ (see \cite[\S5]{Gup} for more on this space). We briefly recall here that a framed representation is a pair of a $\pslc$-representation $\rho$ together with a framing $\beta:F_\infty \to \cp$ which is a $\rho$-equivariant map defined on the lift of $\mathbb{M}$ to the ideal boundary of the universal cover (this set of ideal boundary points is the \textit{Farey set} $F_\infty$, following \cite[Section 1.3]{FG}). 
Such a (framed) monodromy map had been previously defined by Allegretti-Bridgeland \cite{All-Bri} for a subspace $\mathcal{P}^\ast (\sm)$ corresponding to meromorphic projective structures with \textit{no apparent singularities}. Briefly: at a regular singularity, they had defined the framing using the signing to choose a fixed point of the monodromy around it, and at an irregular singularity, they had defined the framing in terms of the asymptotics of the solutions of the Schwarzian equation in Stokes sectors. The latter description was shown to be equivalent to considering the asymptotic values of the developing map in \cite{Gup-Mj} (see \S4.1 of that paper). We extend this here to the case of regular singularities, to provide a more geometric definition of the framing in terms of the asymptotic behaviour of the developing map, that also applies at apparent singularities. 

\vspace{.05in} 

Our next theorem then characterizes the image of this framed monodromy map. Here, $\widehat{\bigchi}^\ast(\sm)$ denotes the subset of non-degenerate framed representations (see Definition \ref{def:deg}) that had been introduced in \cite{All-Bri}. 

\begin{theorem}\label{thm2} The image of $\widehat{\Phi}$ in \eqref{mmap} is precisely the space $\widehat{\bigchi}^\ast(\sm)$ of non-degenerate framed representations. 
\end{theorem}

 Note that Allegretti-Bridgeland had shown that the image of $\mathcal{P}^\ast (\sm)$ is \textit{contained} in $\widehat{\bigchi}^\ast(\sm)$ (see Theorem 6.1 of their paper). We use completely different techniques, applying our (Grafting) Theorem \ref{thm1}, to prove this inclusion for the monodromy map $\widehat{\Phi}$. In particular, this provides an alternative proof of \cite[Theorem 6.1]{All-Bri}. The opposite inclusion, needed to complete the proof of Theorem \ref{thm2} is already known as it follows from the constructions in \cite{Faraco-Gupta} and \cite{Gup-Mj} together with \cite[Theorem 9.1]{All-Bri} which implies that non-degenerate framed representations have Fock-Goncharov coordinates with respect to some ideal triangulation. 

Once again, Theorem \ref{thm2} had been proved in the case that $\mathbb{P}= \emptyset$ in \cite{Gup-Mj_2}, and in the ``opposite" case when $\mathbb{M} = \mathbb{P}$ in \cite{Gup} under the additional assumption of no apparent singularities. See also the recent paper of Nascimento (\cite{Nasci}) which discusses the case of projective structures with Fuchsian-type singularities, and Le Fils (\cite{LF23}) which handles the case of projective structures with branch-points. For a closed surface, the image of the monodromy map was characterized by Gallo-Kapovich-Marden in \cite{GKM}; the case of a punctured surface was left as an open question in that paper. 

\vspace{.05in} 

As an immediate corollary we obtain:

\begin{coro}
    A representation $\rho: \pi_1(\mathbb{S} \setminus \mathbb{P}) \to \pslc$ arises as the monodromy of a meromorphic projective structure in $\mathcal{P}^\pm (\sm)$ (forgetting the framing) if and only if there exists a framing $\beta$ such that the pair $(\rho, \beta)$ is non-degenerate. 
\end{coro}

In \S4 we provide an alternative characterization of the monodromy representations that appear, without involving framings -- see Corollary \ref{cor:reps}. In the case that $\mathbb{S}$ has no boundary components (i.e. $\mathbb{M} = \mathbb{P}$), this coincides with the representations described in Theorem 1.1 of \cite{Faraco-Gupta}. This paper thus provides an alternative proof of that result, sans the construction of affine structures that is discussed in detail in \cite{Faraco-Gupta}.

\vspace{.05in}

Finally, we also prove:

\begin{theorem}\label{thm3} The monodromy map $\widehat{\Phi}$ in \eqref{mmap} is a local homeomorphism.
\end{theorem}

For a closed surface $S$, the fact that the monodromy map from the space of projective structures to the $\pslc$-representation variety 
\begin{equation}\label{pg}
\Phi:\mathcal{P}_g \to \bigchi(S)
\end{equation}
is a local homeomorphism was a classical result of Hejhal in \cite{Hejhal}. This can be considered a special case of the Ehresmann-Thurston Principle concerning the holonomy map from the deformation space of geometric structures on a closed manifold to the corresponding representation-variety (see, for example, \cite{GoldmanICM}). 
This local homeomorphism result was also proved in the case of projective structures with regular singularities having loxodromic monodromy by Luo \cite{Luo} (\textit{c.f}.\ the discussion in \S2.3 of \cite{Gup}). Recently, it was also proved for projective structures with ``cusps" (that allow apparent singularities) in \cite[Theorem 5.6]{Baba3}, and those without apparent singularities and with fixed residues at the poles in \cite{Seran}. The work in \cite{Gup-Mj_2}  had proved the above Theorem in the case that $\mathbb{P}= \emptyset$ (i.e.\ only irregular singularities); our proof here follows their strategy and reduces to an application of a relative version of the Ehresmann-Thurston Principle.

\vspace{.05in}

Using the work of \cite{All-Bri}, we can in fact conclude:

\begin{coro}\label{cor:biholo} The monodromy map $\widehat{\Phi}$ is a local biholomorphism.
\end{coro}

Theorem \ref{thm3} verifies a conjecture by Allegretti in \cite{All2}, and Corollary \ref{cor:biholo} answers a question raised in \cite{All-Bri}.

\vspace{.05in}

In the context of meromorphic projective structures, it remains to \textit{``explore the non-uniqueness of projective structures with given monodromy"}, quoting from \cite[Problem 12.2.1]{GKM}. For a closed surface, the fibers of the monodromy map \eqref{pg} are discrete by Hejhal's result, and have been studied by Baba in \cite{Baba1, Baba2}; it is conceivable that some of the techniques developed there can be generalized to the case of the framed monodromy map $\widehat{\Phi}$.

\vspace{.05in}

Recall that the Schwarzian equation \eqref{schw} is a first-order linear ODE on the Riemann surface. In a broader context, linear differential equations of \textit{higher} order $n\geq 2$ can be considered (see, for example, \cite{Hejh2}). These determine what are called $\slnc$-opers in the literature (see \cite[\S15.6]{FBZ} and the references therein); indeed, projective structures form the special case of $\slc$-opers. Another direction to pursue would be to characterize the surface-group representations into $\pslnc$ that arise as the monodromy of $\text{SL}_n$-opers. In contrast with the $n=2$ case, it follows from a dimension count that for $n>2$ such representations form a subset of positive codimension in the $\pslnc$-character variety, so the analogue of Theorem \ref{thm3} does not hold. The monodromy map has been studied recently by Alley in \cite{Alley}, in a special case of meromorphic {cyclic} $\slnc$-opers on $\cp$.

Note that the case of \textit{first-order} linear ODEs on a Riemann surface is also interesting, as they relate to the theory of translation structures on surfaces associated with abelian differentials. In that context, the monodromy is determined by the periods of the differential, and the analogue of Theorem \ref{thm2} for closed surfaces was classical work of Haupt in \cite{Hau20} that was re-proved by Kapovich in \cite{Kap20}, and extended recently in work 
of Le Fils (\cite{LF20}) and Bainbridge-Johnson-Judge-Park (\cite{BJJP22}). 
For punctured surfaces and meromorphic differentials, an analogue of Theorem \ref{thm2} was proved in recent work of Chenakkod-Faraco-Gupta (\cite{CFG22}), that was extended in  Chen-Faraco (\cite{CF24}).

\vspace{.05in}

\subsection*{Acknowledgements.} The first-named author carried out most of this work while an undergraduate student at the Indian Institute of Science, and is grateful to Kishore Vaigyanik Protsahan
Yojana (KVPY) for their fellowship and contingency grant. The second-named author is grateful to the Department of Science and Technology, Govt.of India for its support via grant no. CRG/2022/ 001822. He also thanks Gianluca Faraco and Mahan Mj for their interest and their collaboration in previous papers that led up to this.   Both authors would like to thank Dylan Allegretti and Lorenzo Ruffoni for helpful correspondences, and the anonymous referee for helpful comments. This work is also supported by the DST FIST program - 2021 [TPN - 700661].

\section{Signed parameter Spaces}


As mentioned in the Introduction, let $(\sm)$ be a marked and bordered surface of genus $g$ and $k$ boundary components, and let $\mathfrak{n}=\{n_1,n_2,...,n_k\}$ with $n_i\geq 3$ be the associated integer-tuple. Here recall that $(n_i-2)$ is the number of marked points on the $i$-th boundary component, for each $1\leq i\leq k$.  We shall also denote by $m=\lvert \mathbb{P}\rvert$ the number of marked points in the interior of $\mathbb{S}$. 
We define  
\begin{equation}\label{eq:chi} 
    \chi=6g-6+\sum_{i=1}^k(n_i+1)+3m
\end{equation}
and we shall assume this is positive, throughout this article. This equivalent to requiring that if $g=0$, then $\lvert \mathbb{M} \rvert \geq 3$.

\begin{subsection}{The Enhanced Teichm\"{u}ller Space}  We shall define the space of hyperbolic structures on the marked and bordered surface $(\sm)$. Such a hyperbolic structure $X$ will have either a cusp or a geodesic boundary component at each interior puncture $\mathbb{P} \subset \mathbb{M}$, and each boundary component of $\mathbb{S}$ with marked points will be a ``crown end" (see \cite[\S3.2]{Gup-Mj}) where each boundary arc between marked points is a bi-infinite geodesic (a \textit{side} of the crown).  

\begin{defn}
A marking of a hyperbolic surface $X$  as above by $(\mathbb{S}, \mathbb{M})$ is a homeomorphism $f: \mathbb{S}\backslash\mathbb{M}\rightarrow X^\circ$, where $X^\circ$ is the surface obtained from $X$ by removing the boundary components homeomorphic to $S^1$. Note that such a marking must take the set of interior punctures $\mathbb{P}$ to geodesic ends or cusps, and boundary arcs between marked points to bi-infinite geodesics of the crown ends. Two such markings $(f_1,X_1)$ and $(f_2,X_2)$ are defined to be equivalent if there is an isometry $g:X_1 \rightarrow X_2$ such that $g\circ f_1$ is isotopic to $f_2$ relative to the marked points $\mathbb{M}$. Then, the set of such markings $(f,X)$ under this equivalence relation forms the  Teichm\"{u}ller space $\mathcal{T}(\mathbb{S},\mathbb{M})$.
\end{defn}

Now, we want to provide an additional signing parameter to our hyperbolic surfaces, namely an orientation to the boundary components homeomorphic to $S^1$. For each component, this parameter is a choice of an element in $\{\pm 1\}$, depending on whether it agrees with the orientation induced from that on the surface.  Recall that there are at most  $m$ such boundary components. Thus, we get:
\begin{defn}[Definition 3.4 of \cite{All}]
 The enhanced Teichm\"{u}ller space $\mathcal{T}^{\pm}(\mathbb{S},\mathbb{M})$ is a $2^m$-fold branched cover of $\mathcal{T}(\mathbb{S},\mathbb{M})$ (branched on the set of surfaces with at least one cusp end), with branching data given by a choice of a sign at each geodesic end.
\end{defn}
 Parametrizing this space with shear coordinates, we have the following:

\begin{theorem}[Proposition 3.5. of \cite{All}]\label{dim_teichmuller}
For a pair $(\mathbb{S},\mathbb{M})$ with $|\mathbb{M}|\geq3$ if $g(\mathbb{S})=0$, we have that $\mathcal{T}^{\pm}(\mathbb{S},\mathbb{M})$ is homeomorphic to $\mathbb{R}^{\chi}$. 
\end{theorem}

(Here the expression for $\chi$ is given in \eqref{eq:chi}.) 

\vspace{.1in} 

We denote an element of this set by a tuple $(X,f,\sigma)$, where $X$ is the hyperbolic surface, $f:\mathbb{S}\backslash\mathbb{M}\rightarrow X^{\circ}$ is the marking, and $\sigma$ is the signing of each geodesic boundary component of $X$. 

\end{subsection}

\begin{subsection}{The Enhanced Space of Measured Laminations}\label{enhanced_space}

Let $(X,f,\sigma)$ be an element of $\mathcal{T}^{\pm}(\mathbb{S},\mathbb{M})$ as above. Now, we define the space of measured geodesic laminations on this element. First, we define a geodesic lamination on the surface:
\begin{defn}
A geodesic lamination on $X$ is a closed subset $L$ in $X$ that is a disjoint union of simple closed geodesics or a bi-infinite geodesics on $X$, including the boundary geodesics of $X$.
\end{defn}
 Since we always include the boundary geodesics, it will be helpful to consider $L$ as the disjoint union of two subsets, $L^{\circ}\in X^{\circ}$ and $\partial X$. Since the geodesics in a lamination cannot meet the boundary geodesics, it follows that the leaves of the lamination must satisfy the following at the neighbourhood of an end of $X$:
\begin{itemize}
    \item If the end is a cusp, the leaves of the geodesic exiting the end must go inside the cusp, without any spiralling. This can be seen by lifting to the universal cover $\mathbb{H}^2$; if we assume that the cusp end $C$ lifts to a horodisk $H = \{z\ \vert\ \text{Im}(z) > h\}$ in the upper half-plane model of $\mathbb{H}^2$, such that $C = H/\langle z \mapsto z+1 \rangle$, then any such geodesic leaf lifts to a vertical line that remains in a single lift of the fundamental domain. 
    
    \item If the end is a geodesic end, then each leaf entering a neighbourhood of the end must spiral in a direction around the geodesic and accumulate on it. Since the leaves are disjoint, it follows that (if there is more than one leaf) they must all spiral in the same direction.
    \item If the end is a crown end, similar to a cusp, each leaf entering the end must go into one of the cusps of the crown. 
\end{itemize}
We recall another useful Lemma regarding geodesic laminations:
\begin{lem}
Given a geodesic lamination $L$ on a hyperbolic surface $X$, there is a $\pi_1(X)$-invariant ideal triangulation of the universal cover $\widetilde{X}$ such that no leaf of the lamination intersects the interior of a triangle.
\end{lem}
\begin{proof}
The completion of the complement of the geodesic lamination $L$ comprises finitely many connected hyperbolic surfaces that have crowns, cusps or geodesic boundary (see, for example \cite[pg. 7]{Bonahon}). One can equip each of these components with an ideal triangulation, where the (finitely many) edges of the triangulation are geodesic lines between cusps or crown boundary cusps or spiralling onto geodesic boundary components. The lift of these geodesic lines, together with the lift of $L$, defines a $\pi_1(X)$-invariant ideal triangulation of the universal cover.
\end{proof}

\begin{lem}\label{lem:lam}
Given a geodesic lamination $L$, for each end of $X$ there is a neighbourhood of that end such that only finitely many leaves of the lamination enter it. For a crown end, we can take the neighbourhood to be the whole crown.
\end{lem}
\begin{proof}
Consider a triangulation of the universal cover of $X$ as before. Then, the triangulation descends to the surface itself, due to $\pi_1(X)$-equivariance. From the Gauss-Bonnet formula, the number of triangles on the surface is exactly $-2\chi(X)$. Now, each triangle can contribute at most $3$ geodesics going into the cusps or crown cusps of $X$. Since each geodesic going into such an end is adjacent to two triangles, it follows that the total number of geodesics going into cusps and crown cusps is at most $-2\chi(X)\times 3 = -6\chi(X)$. This proves the first statement. For the second, note that any leaf of the lamination intersecting a crown must be a geodesic going to one or two of its boundary cusps (i.e.\ crown tips), so the statement follows.
\end{proof}

We briefly recall the notion of a  measured lamination, which is a geodesic lamination $L$ as above, together with a transverse measure on it (for details see, for example, \cite[\S 1.8]{Pen}):
\begin{defn}
Given a geodesic lamination $L$, let $\Lambda(L)$ denote the collection of all compact $1$-manifolds embedded in $X$ which are transverse to $L$ and such that their boundary (if it exists) lies in $X\backslash L$. Then, a measure on $L$ refers to a function $\mu: \Lambda(L) \rightarrow \mathbb{R}_{\geq 0}$ such that it is transverse to $L$ (i.e. if $\alpha$ and $\beta$ are 1-manifolds that are homotopic via 1-manifolds with boundary in $X\backslash L$, then $\mu(\alpha)=\mu(\beta)$), $\sigma$-additive (i.e.\ if $\alpha=\cup_{i\in \mathbb{N}} \alpha_i$ with $\alpha_i \cap \alpha_j = \partial\alpha_i \cap \partial \alpha_j$, then $\mu(\alpha)=\sum\limits_i\mu(\alpha_i)$), and its support is $L^{\circ}$. 
This pair $(L,\mu)$ defines a measured lamination on $X$.
\end{defn}

Note that given a measured lamination, every isolated curve in $L$ obtains a weight. 
We can define a topology on the set of measured laminations as follows: First, given a lamination, we can lift the lamination to the universal cover to get a $\pi_1(X)$-invariant set of geodesics in $\mathbb{H}^2$. Thus, it gives us a subset of the space $M_{\infty} = (\partial \mathbb{H}^2 \times \partial \mathbb{H}^2 \backslash \Delta)/\sim $, where $\sim$ is the equivalence relation $(x,y)\sim(y,x)$. Then, a measure on the lamination gives us a Borel measure on the space $M_\infty$. We define the topology on the set of measured laminations to be the topology it inherits from the weak-$*$ topology of measures on $M_\infty$. Thus, we have:

\begin{defn}
The space of measured laminations on $X$, denoted $\mathcal{ML}(X)$, is the set of all pairs $(L,\mu)$ defined above, endowed with the topology it inherits from being a subset of the space of measures on $M_\infty$, endowed with the weak-$*$ topology.
\end{defn}


Now, we can parametrize the space $\mathcal{ML}(X)$ via the use of Dehn-Thurston coordinates as described in the unpublished notes of Dylan Thurston \cite{Thu}, to get the following (\textit{c.f.} Proposition 1.5. of \cite{Hat}):

\begin{prop}\label{dim_measured_laminations}
The space $\mathcal{ML}(X)$ is homeomorphic to $\mathbb{R}^{\chi-p}\times\mathbb{R}_{\geq0}^{p}$, where $p$ is the number of cusp ends of $X$.
\end{prop}
\begin{proof}
Recall here that $\chi$ is given by \eqref{eq:chi}. 
For simplicity, let us first assume that there are no crown ends on $X$, i.e.\ $k=0$ or $\mathbb{M} = \mathbb{P}$. In this case $\chi = 6g-6 + 3m$ where $m=\lvert \mathbb{M} \rvert$. 

We construct a pants decomposition of the topological surface underlying $X$; an Euler characteristic count this gives us a total of $\#P= 2g+m-2$ pairs of pants. Hence, there are $t=\frac{3\cdot\#P-m}{2}$ simple closed  curves in the interior of $X$ that gives the decomposition. To this set of curves, we also add in the geodesic boundaries and peripheral loops around the cusps of $X$ to obtain a total of $t+m$ curves. Let these curves be given by $(P_1,P_2,...,P_t,P_{t+1},...,P_{t+m})$. Also, let us choose a collection of dual curves $D_i$ for each $P_i$.
Then, we have that associated to each pants curve $P_i$, there is a pair of parameters for the measured lamination, namely the length or intersection parameter $i(\mu,P_i)$ and the twist parameter $\theta(\mu,P_i)$, the latter measured using $D_i$. It follows from the discussion in \cite{Thu} that the space of measured laminations is homeomorphic to this  space of intersection and twist parameters. This gives a factor of $\mathbb{R}^{2t}$.

For a pants curve in the interior of $X$, the length parameter can take values in $\mathbb{R}_{\geq 0}$ while the twist parameter takes values in $\mathbb{R}$. Considering the pairs of parameter values $(i, \theta) \in \mathbb{R}_{\geq 0} \times \mathbb{R}$, observe that in the case of zero length, the twist parameter does not matter, leading to the identification $(0,\theta)\sim(0,-\theta)$. Hence, the parameter space is in fact homeomorphic to $\mathbb{R}^2$ for an interior curve. For a pants curve that is a boundary geodesic, the twist parameter can only take $\pm \infty$ as values if the length parameter is not zero. If the length parameter is zero, the twist parameter is also necessarily zero. Hence, it follows that the space parametrizing pairs $(i,\theta)$ for a geodesic boundary is homeomorphic to $\mathbb{R}$. This gives a factor of $\mathbb{R}^b$, where $b$ is the number of geodesic boundary ends of $X$. For a pants curve that corresponds to a cusp end, hence there is no twist parameter. Thus the corresponding space of parameters is homeomorphic to $\mathbb{R}_{\geq 0}$. This gives a factor of $\mathbb{R}_{\geq0}^p$. Note that $p+b = m$; collecting the factors, we get a space homeomorphic to $\mathbb{R}^{2t}\times\mathbb{R}^b\times\mathbb{R}_{\geq 0}^p \equiv \mathbb{R}^{\chi-p}\times\mathbb{R}_{\geq 0}^p $, as desired. This handles the case when there are no crown ends. 

The case of (only) crown ends was established in \cite[Proposition 3.8]{Gup-Mj}, and we now follow the arguments there to extend the above proof to include crown ends. As in that proof, we can first remove the crown-ends to obtain a surface with $k$ additional boundary components. Following the argument above, the space of  measured laminations on the resulting surface $X^\prime$  with $(b+k)$ boundary components and $p$ cusps is homeomorphic to $\mathbb{R}^{2t}\times\mathbb{R}^{b+k}\times\mathbb{R}_{\geq 0}^p$. By \cite[Proposition 3.7]{Gup-Mj}, the space of measured laminations on each crown is homeomorphic to $\mathbb{R}^{n_i-1}$ where $(n_i-2)$ is the number of boundary cusps of the $i$-th crown. Moreover, these measured laminations can be glued to a measured lamination on $X^\prime$ as long as one parameter -- the transverse measure of each boundary component that is glued -- matches. Thus, the space of measured laminations on $X$ is homeomorphic to $\mathbb{R}^{2t}\times\mathbb{R}^{b+k}\times\mathbb{R}_{\geq 0}^p \times \prod\limits_{i=1}^k \mathbb{R}^{n_i-2} \equiv \mathbb{R}^{\chi-p}\times\mathbb{R}_{\geq 0}^p $, as desired. \end{proof}

Now, we consider an additional signing parameter for elements of  the space $\mathcal{ML}(X)$: At each end of $X$ that is a cusp, we assign a sign $\pm$ to the weight of measured lamination entering the end. By this marking, we get a branched cover \[
\mathcal{ML}^{\pm}(X) \rightarrow \mathcal{ML}(X)
\] 
branched over those measured laminations which have at least one cusp end with no weight of the lamination entering it. 
 \begin{lem}\label{dim_measured}
     The space $\mathcal{ML}^{\pm}(X)$ is homeomorphic to the cell $\mathbb{R}^{\chi}$.
 \end{lem}
\begin{proof}
    This follows from the proof of Proposition \ref{dim_measured_laminations}, by observing that at each cusp end, the non-negative parameter, together with a sign, can be thought of as taking values in  $\mathbb{R}$.  
\end{proof}
Now, note that the spaces $\mathcal{ML}^{\pm}(X)$ are canonically homeomorphic to each other as $X$ ranges over $\mathcal{T}(\mathbb{S},\mathbb{M})$. This motivates us to define a space parametrising signed measured laminations on the marked surface $(\mathbb{S},\mathbb{M})$. We note that the space parametrising measured laminations for the case $\mathbb{P}=\varnothing$ was introduced in \S 3.3 of \cite{Gup-Mj}.
\begin{defn}
    The space of signed measured laminations, denoted $\mathcal{ML}^\pm(\mathbb{S},\mathbb{M})$, is a space parametrising measured laminations on $(\mathbb{S},\mathbb{M})$ along with a choice of signings for the total incident weights at the set of punctures $\mathbb{P}\subset\mathbb{M}$. It has a topology induced from the transverse measures on finitely many closed curves on the surface, along with the measures of arcs crossing into the cusp ends for the interior punctures and the crowns.
\end{defn}
Then the previous lemma can be interpreted as proving: 
\begin{prop}
    The space $\mathcal{ML}^\pm(\mathbb{S},\mathbb{M})$ is homeomorphic to the cell $\mathbb{R}^{\chi}$.
\end{prop}



\end{subsection}

\begin{subsection}{The Space of Signed Projective Structures}\label{signed_proj_struct}
For a surface with a marking $(\mathbb{S},\mathbb{M})$ as before, Allegretti-Bridgeland introduced the space $\mathcal{P}(\mathbb{S},\mathbb{M})$ in \cite{All-Bri}. This space parametrises meromorphic projective structures on $\Sigma_g$ with $k$ poles of orders given by the tuple $\mathfrak{n}$ and $m$ poles of order less than or equal to $2$, which are marked by the pair $(\mathbb{S},\mathbb{M})$. We also have the following additional structure on this space:

\begin{theorem}[Proposition 8.2 of \cite{All-Bri}]
The space $\mathcal{P}(\mathbb{S},\mathbb{M})$ has the natural structure of a complex manifold of complex dimension $\chi$, hence of real dimension $2\chi$.
\end{theorem}

Similar to the definition of the space of signed measured laminations, we shall define a signed version of the space $\mathcal{P}(\mathbb{S},\mathbb{M})$, with the signing being given at each puncture $\mathbb{P}$ by the exponent of the projective structure at the puncture (defined below). The discussion here follows that in \S 8.2 of \cite{All-Bri}, though we shall provide an alternative description in terms of a fiber-product  (Lemma \ref{fiber_product}). 

\subsubsection{The Exponent at a Regular Singularity} Given a meromorphic projective structure, we can define the \textit{leading coefficient} at a regular singularity $p$ as follows:
\begin{equation*}
a_p := \displaystyle \lim_{z\rightarrow p}  q(z)\cdot z^2
\end{equation*}
where $q(z)dz^2$ is the Schwarzian derivative of the projective structure in a uniformizing neighbourhood of the puncture. Note that this is independent of the choice of the uniformizing chart, hence well-defined for the projective structure.  In other words, in any coordinate $z$ around the regular singularity, the quadratic differential $q$ has the form:
\begin{equation*}
q(z) = \left(\frac{a_p}{z^2} + \cdots \right) dz^2 
\end{equation*}

\vspace{.05in} 

We then define:

\begin{defn}[Exponent]\label{defn:exp}  The exponent at a regular singularity  $p\in \mathbb{P}$ of a meromorphic projective structure is the complex number 
\begin{equation}\label{expon}
r_p := \pm 2\pi i \sqrt{1-2a_p}
\end{equation}
defined up to sign, where $a_p$ is the leading coefficient at $p$. 
\end{defn}

\noindent
 \textit{Remark.} Our definition of exponent slightly differs from the one in \cite{All-Bri}, in that they have $(1+4a_p)$ under the square root in place of $(1-2a_p)$. This is only matter of convention, arising from the fact that their Schwarzian equation has a constant factor of $-1$ of the zeroth order term while ours (in \eqref{schw}) has constant $\frac{1}{2}$.

\subsubsection{The Signed Space of Projective Structures}
First, we associate a signing to our marked projective structures, as in \S 3.5 of \cite{All-Bri}:
\begin{defn}
    A signed marked projective structure is a marked projective structure on the pair $(\mathbb{S},\mathbb{M})$ along with a signing at each of the regular singularities, i.e.\ a choice of sign for the exponent at each regular singularity.
\end{defn}
Now, we define the space of signed marked projective structures as a cover of $\mathcal{P}(\mathbb{S},\mathbb{M})$, following Proposition 8.4 of \cite{All-Bri}:
\begin{theorem}
    There exists a complex manifold $\mathcal{P}^\pm(\mathbb{S},\mathbb{M})$ that parametrises signed, marked projective structures on $(\mathbb{S},\mathbb{M})$, along with a finite branched covering map \[
    \mathcal{P}^\pm(\mathbb{S},\mathbb{M}) \rightarrow \mathcal{P}(\mathbb{S},\mathbb{M})
    \]
    of degree $2^{|\mathbb{P}|}$, obtained by forgetting the signing.
\end{theorem}

\begin{proof}
The proof is essentially the same as that in \cite{All-Bri}. First, we define the map
    \[
a : \mathcal{P}(\mathbb{S},\mathbb{M}) \rightarrow \mathbb{C}^{\lvert \mathbb{P}\rvert }
\]
sending each projective structure to the collection of leading coefficients at each of its regular singularities. We have that $a$ is a holomorphic map because of the way the complex structure on $\mathcal{P}(\mathbb{S},\mathbb{M})$ is constructed (see, for example, \cite[Proposition 7.3]{All-Bri}). Moreover, $a$ is a submersion, by a similar argument as in the proof of Lemma 6.1 in \cite{BS15} -- it suffices to construct a quadratic differential whose leading coefficient is nonzero at a given regular singularity and zero at all other regular singularities, which is clear from an application of Riemann-Roch.

Due to these properties of $a$, we can construct a $2^{\lvert\mathbb{P}\rvert}$-branched cover $\mathcal{P}^{\pm}(\mathbb{S},\mathbb{M})$ of $\mathcal{P}(\mathbb{S},\mathbb{M})$, branched over the zero loci of $\{1-2a_p\}_{p\in\mathbb{P}}$, and such that the points in a fiber of the branching represent a choice of a sign (i.e. an element of $\{\pm 1\}$)  for the non-zero exponents at the regular singularities.
\end{proof}

\noindent\textit{Remark.} We note that \cite{All-Bri} do the same construction as above, the only difference being that they define the covering space over the subset of projective structures without any apparent singularities. 

\subsubsection{Describing the Signed Space by Fiber Products}
We provide another description of the space $\mathcal{P}^{\pm}(\mathbb{S},\mathbb{M})$ via fiber products, that will be helpful in defining the grafting map later. Recall the definition of a fiber product (for example, see Chapter 1, \S 11 of \cite{Lang}):

\begin{defn}\label{def:fiberprod}
    Let  $\mathcal{C}$ be a category. Suppose we are given two morphisms $f : A\rightarrow C$ and $g: B\rightarrow C$ among objects $A,B,C$. Then, the fiber product of $A$ and $B$ with respect to the given morphisms is an object $A\times_CB$ in $\mathcal{C}$, along with morphisms $g' : A\times_CB \rightarrow A$ and $f' : A\times_CB \rightarrow B$, such that $f\circ g' = g \circ f'$, and the following universal property holds: given any object $D$ with morphisms $f'':D\rightarrow B$ and $g'' : D\rightarrow A$ such that $f\circ g''=g\circ f''$, there is a unique morphism $i : D \rightarrow A\times_CB$ such that $f''=f'\circ i$ and $g''=g' \circ i$, i.e$.$ the maps $f'',g''$ factor through $i$.
\end{defn}

Now, we have the following:
\begin{lem}\label{fiber_product}
    Let $m = \lvert \mathbb{P}\rvert$ and consider the map \[
    r^2: \mathcal{P}(\mathbb{S},\mathbb{M}) \rightarrow \mathbb{C}^{m}
    \]
    that assigns to a meromorphic projective structure the set of squares of exponents (see \eqref{expon}) at its regular singularities $\mathbb{P}$, and \[
    sq : \mathbb{C}^{m} \rightarrow \mathbb{C}^{m}
    \]
    that squares each coordinate of $\mathbb{C}^{m}$, i.e$.$ $sq(z_1,z_2,...,z_m)=(z_1^2,z_2^2,...,z_m^2)$. Then, $\mathcal{P}^{\pm}(\mathbb{S},\mathbb{M})$ is isomorphic to the fibre product $\mathcal{P}(\mathbb{S},\mathbb{M})\times_{\mathbb{C}^{m}}\mathbb{C}^{m}$, in the category of complex manifolds as well as in the category of topological spaces. 
    \[\begin{tikzcd}
	{\mathcal{P}^{\pm}(\mathbb{S},\mathbb{M})} & {\mathbb{C}^{\lvert \mathbb{P}\rvert}} \\
	{\mathcal{P}(\mathbb{S},\mathbb{M})} & {\mathbb{C}^{\lvert \mathbb{P}\rvert}}
	\arrow["r"', from=1-1, to=1-2]
	\arrow["\pi", shift right=2, from=1-1, to=2-1]
	\arrow["{r^2}", from=2-1, to=2-2]
	\arrow["sq"', from=1-2, to=2-2]
\end{tikzcd}\]
\captionsetup{justification=centering}
\captionof{figure}{Fiber product defining the space $\mathcal{P}^{\pm}(\mathbb{S},\mathbb{M})$. Here, $\pi$ is the forgetful projection map.}
\end{lem}

Although this is just a re-statement of the way the branched cover is defined, and is in fact implicit in the construction described in \cite[\S8.2]{All-Bri}, it is convenient as we shall use the universal property of fiber products in \S3.3 while defining the signed grafting map. 

\vspace{0.1in}

\noindent\textit{Remark.} Note that the exponent at regular singularities was defined only up to sign on the space of unsigned projective structures. The construction of the signed space allows us to uniquely define the exponent at each regular singularity of a signed projective structure. Thus, we get a well-defined holomorphic map $r: \mathcal{P}^{\pm}(\mathbb{S},\mathbb{M}) \rightarrow \mathbb{C}^m$ which sends a signed projective structure to the exponents at each of its regular singularities. This is the map on the upper side of the commuting square in Figure 1.

\end{subsection}
\section{The Grafting Theorem}

Our first result concerns a grafting parametrization for the space of signed projective structures $\mathcal{P}^{\pm}(\mathbb{S},\mathbb{M})$. If one ignores signings, the grafting operation was briefly described in the Introduction; for details, see the references mentioned there, and for the present context of marked and bordered hyperbolic surfaces, we refer the reader to \cite{Gup-Mj} and \cite{Gup}. 

\vspace{.1in}

We begin by stating the following general result concerning simply-connected projective surfaces, essentially due to Kulkarni-Pinkall, see \cite[Theorem 10.6]{Kul-Pin}, and also Theorem 2.1. of \cite{Gup-Mj}:

\begin{theorem}\label{thm: kulpink}
Let $\tilde{X}$ be a simply-connected projective surface that is not projectively
isomorphic to $\mathbb{C}$, or the universal cover of $\mathbb{CP}^1 \backslash \{0,\infty\}$. Then there exists a unique
measured lamination $L$ on the hyperbolic plane $\mathbb{H}^2$ such that $\tilde{X}$ is obtained by grafting $\mathbb{H}^2$ along $L$. The map associating $L$ to $\tilde{X}$ is equivariant, i.e. if $\tilde{X}$ is the universal cover of a
projective surface $X$, and the developing map $\tilde{X} \rightarrow \mathbb{CP}^1$ is $\rho_\mathbb{C}$-equivariant for a representation $\rho_\mathbb{C} : \pi_1(X) \rightarrow \mathbb{PSL}_2(\mathbb{C})$, then $L$ is invariant under the image of a naturally associated
representation $\rho_\mathbb{R} : \pi_1(X) \rightarrow \mathbb{PSL}_2(\mathbb{R})$. Moreover, the image $\Gamma$ of $\rho_\mathbb{R}$ is discrete, and the
quotient $\mathbb{H}^2/\Gamma$ is homeomorphic to $X$.
Finally, the mapping $\tilde{X} \mapsto L$ is continuous.
\end{theorem}

We refer the reader to the sketch of the proof provided in \cite{Gup-Mj}. However, it will help to keep in mind the broad strategy of the proof: Each point $x$  in the universal cover $\widetilde{X}$ is contained in  ``maximal disk"  $U_x$ whose image under the developing map of the projective structure  is a round disk $V_x$  in $\cp$. The  convex hull $C(V_x)$ in $\mathbb{H}^3$ is bounded by a totally-geodesic hyperbolic plane. The envelope of these convex hulls, as $x$ varies in $\widetilde{X}$, then defines a $\rho_\mathbb{C}$-equivariant ``pleated plane" in $\mathbb{H}^3$, bent along an equivariant collection of geodesic pleating lines.  In fact, each convex hull $C(\overline{U_x} \cap \partial _\infty{\widetilde{X}})$ maps to a totally-geodesic face or ``plaque" of the pleated plane, or a pleating line. ``Straightening" the pleated plane then yields a totally-geodesic copy of $\mathbb{H}^2$, on which the pleating lines define the measured lamination $L$.  This equivariant pleated plane in $\mathbb{H}^3$ determined by the projective structure is a key intermediate object, interesting in its own right, that we shall refer to later. 

\vspace{.05in} 

The above result will be crucial in the proof of Theorem \ref{thm1}; in particular, we shall apply it to the lift of a given projective structure on $\mathbb{S}$ to its universal cover.

 \begin{subsection}{Grafting a marked surface along a measured lamination} 
Given an element $X\in \mathcal{T}^\pm(\mathbb{S},\mathbb{M})$ and a measured geodesic lamination $\lambda$ on $X$, we can perform the operation of grafting the surface along this measured lamination $\lambda$, to obtain a projective structure on $\mathbb{S}$.  In this section we first show that this projective structure is in fact in $\mathcal{P}(\sm)$, i.e. the Schwarzian derivative of the developing map descends to a quadratic differential with poles of prescribed orders at the punctures of the underlying Riemann surface.   (For brevity we shall often abbreviate this  by saying that the projective structure has poles of prescribed orders.) 
It suffices to verify this for each end of $X$; recall that there are only finitely many leaves of the lamination $\lambda$ going into a cusp, geodesic end or crown end. 

\vspace{.05in} 

For a boundary component of $\mathbb{S}$ (which defines a  crown end of $X$), we have the following from \cite{Gup-Mj}:

 \begin{lem}\label{irreg_graft}[Proposition 4.2. of \cite{Gup-Mj}]
The operation of grafting along a measured geodesic lamination exiting a crown end produces a meromorphic projective structure with a pole of order $n_i$ on the underlying punctured Riemann surface.
 \end{lem}
 
 For the cusps and geodesic boundary components of $X$, we have the following Lemma. For the computations in the proof, it will help to recall the definition of the Schwarzian derivative:
 \begin{equation}\label{schw-deriv}
S(f) = \left( \frac{f^{\prime\prime}}{f^\prime}\right)^{\prime}  - \frac{1}{2} \left( \frac{f^{\prime\prime}}{f^\prime}\right)^{2}. 
\end{equation}
 Although this is essentially Proposition $3.5.$ of \cite{Gup}, we provide a complete proof here that clarifies some of the arguments there; we shall refer to some of the computations throughout this paper.

  \begin{lem}\label{regular_grafting}
The operation of grafting at cusps and geodesic ends produces a projective structure such that the Schwarzian derivative of the developing map on the underlying Riemann surface has a pole of order at most $2$. 
 \end{lem}
\begin{proof}
We shall consider the two cases of a cusp end  and a geodesic end separately:
\begin{itemize}
    \item For a cusp, we can assume that it has a punctured-disk neighbourhood $\mathbb{D}^\ast$  that lifts to $H=\{z : \textrm{Im}(z)>a\}$, with $\mathbb{D}^\ast$  biholomorphic to $H/\langle z\rightarrow z+1\rangle$. Let the lifts of the finitely many geodesics entering the cusp be given by $\{z : \textrm{Re}(z)=a_j\}$ with corresponding weights $\alpha_j$ for $j=1,2,\dots,r$ (Assuming without loss of generality that $0\leq a_1 <a_2<\dots<a_r<1$). Let us denote $\omega_j=e^{i\alpha_j}$.  Define by $E_i(w)$ the elliptic element $z\mapsto \omega_i^{-1} (z-w) +w$ (i.e. clockwise rotation by $\alpha_j$ about the endpoints $\infty$ and $w$), and let $T$ denote the translation $z\mapsto z+1$.

    \vspace{.05in}

    Now, referring to \cite[Lemma 5.5]{Dum}, the monodromy after bending will be conjugate to the element $E_1(a_1)\circ E_2(a_2) \circ ... \circ E_r(a_r)\circ T$, which is easily seen to be 
    \begin{equation}\label{mon-exp}
    z\mapsto \omega_1^{-1}\omega_2^{-1}\dots\omega_r^{-1} z+c
    \end{equation}
    where the constant 
    \begin{equation}\label{cdef} 
        c=a_1+\sum\limits_{i=1}^r \omega_1^{-1}\omega_2^{-1}\dots\omega_i^{-1}(a_{i+1}-a_i)
    \end{equation}
   where we set $a_{r+1}=1$. This element is elliptic if $\alpha :=\sum \alpha_i$ is not a multiple of $2\pi$. Otherwise, this is either a parabolic element or the identity.

    \vspace{.05in}

    To determine the developing map and compute  the Schwarzian derivative, we divide into the following cases. Note that the possible developing maps at a regular singularity are classified by studying solutions of the Schwarzian equation \eqref{schw}; see \S4.1.1 for a discussion. 
    \begin{enumerate}
        \item \textbf{$\alpha$ is not an integer multiple of $2\pi$:} Recall that the operation of grafting introduces ``lunes" (regions in $\cp$ bounded by circular arcs) at every lift of a geodesic leaf entering  $\mathbb{D}^\ast$. The total sum of the angles of lunes is $\alpha$.  Recall that the resulting peripheral monodromy after grafting is an elliptic rotation of angle $\alpha$; we  can assume that it fixes the point $\infty \in \overline{\mathbb{R}} \subset \cp$.  Let $F$ be a fundamental domain of the $\mathbb{Z}$-action on $H$, bounded by  the vertical lines $\{Re(z)=0\} $ and $\{Re(z)=1\}$. On $F$ the developing map is a conformal map, that takes these two sides of $F$ to two circular arcs incident at $\infty$ that differ by an elliptic rotation of angle $\alpha$ fixing $\infty$. Such a conformal map is the map  
        $z\mapsto e^{-2\pi i \alpha z}$; indeed, this takes the two boundary lines of $F$ to the circular arcs $\{e^{2\pi x} : x \in \mathbb{R}^+\}$ and $\{e^{2\pi i \alpha}\cdot e^{2\pi x} : x \in \mathbb{R}^+\}$.  On the punctured disk $\mathbb{D}^\ast$, the developing map  descends to the map $f: w\mapsto w^{-\alpha/2\pi}$, since the universal covering map from $H$ to $\mathbb{D}^\ast$ is $z\mapsto w := e^{2\pi i z}$. 
        

        Moreover, the Schwarzian derivative of the developing map on $F$ descends to the Schwarzian derivative of $f$ on $\mathbb{D^\ast}$.  Computing this Schwarzian derivative using \eqref{schw-deriv}, we obtain $S(f) = \frac{4\pi^2-\alpha^2}{8\pi^2w^2}dw^2$. Since $\alpha \neq \pm 2\pi$,  this quadratic differential has a pole of order $2$ at the puncture.

        \item \textbf{$\alpha$ is an integer multiple of $2\pi$:} Let $\alpha = 2n\pi, n\geq 0$. In the case that $c=0$, note that $n >0$, and we have, from the same discussion as above, that the developing map descends to a map of the form $w\mapsto w^n$ in a neighborhood of the puncture, and hence the Schwarzian derivative is $\frac{1-n^2}{2w^2}dw^2$, which has a pole of order $2$ at the puncture if $n\neq 1$, and a zero at the puncture if $n=1$. 
        
        Henceforth, let us assume that $c\neq 0$, hence the monodromy of the structure around the puncture is a parabolic element.
        In this case, the developing map descends to the map $f :w \mapsto w^{-n}+\textrm{log}(w)$ on the punctured disk. On the universal cover $H$, the developing map is $z\mapsto e^{- 2\pi in z}+2\pi i z$. Indeed, this takes the lines $\{Re(z)=0\}$ and $\{Re(z)=1\}$ to the circular arcs $\{e^{2\pi n x}-2\pi x : x \in \mathbb{R}^+\}$ and $\{e^{2\pi n x} - 2\pi x + 2\pi i  : x \in \mathbb{R}^+\}$ in $\cp$ both incident at $\infty$, and having the same tangential direction there. The monodromy around the puncture in this case is given by $z \mapsto z + 2\pi i $. As before the developing map ``wraps" the infinite strip $\{ 0\leq Re(z)\leq 1 \}$ on $\cp$ by $2n\pi$. 
        
        Computing the Schwarzian derivative of $f$, we obtain  $$S(f) = \frac{w^{2n}-w^n(2n^3+2n)-n^2(n^2-1)}{2w^2(w^n-n)^2} dw^2.$$ If $n=0$, this equals $\frac{1}{2}w^{-2}dw^2$, hence it has a pole of order $2$. For $n\geq1$, we can expand near $0$ to get $(\frac{1-n^2}{2}w^{-2} - 2nw^{n-2} + O(w^{2n-2})  )dw^2$. Clearly, it has a pole of order $2$ if $n\geq 2$. If $n=1$, then it has a pole of order $1$. (This final observation results in the next Corollary \ref{cor:simp}.)
    \end{enumerate}

   \vspace{.05in} 
   
    \item For a geodesic boundary end, we can take the  universal cover of its neighbourhood to be $B=\{z : \frac{\pi}{2}-\epsilon<\textrm{arg}(z)<\frac{\pi}{2}\} \subset \mathbb{H}^2$ with the line $\{\textrm{Re}(z)=0\}$ mapping onto the geodesic. Here, we are assuming that the lift of the surface lies to the right of this geodesic line. The end is biholomorphic to $B/\langle z \rightarrow \lambda z \rangle$, where $\textrm{log}(\lambda)$ is the length of the geodesic end. Now, note that the geodesics entering the end can spiral in one of two directions, clockwise or anticlockwise. Correspondingly, the geodesics either have endpoints at $0$ and the positive real axis, or at $\infty$ and the positive real axis, respectively. 

        \vspace{.05in}

    Let us first consider the case of geodesics spiralling anticlockwise into the boundary component. Let the geodesics entering the end of a fundamental domain $U := \{ z \in \mathbb{H}^2 : 1\leq \textrm{Re}(z) < \lambda \}$ be given by $\{\gamma_i: Re(z)=a_i \}$, assuming $1\leq a_1<\dots<a_r<\lambda$. Recall the definitions of $\alpha_j$, $\omega_j$ and $\alpha$; $\alpha_j$ are the weights of the grafting geodesics, $\alpha = \sum \alpha_j$ is the total weight, and $\omega_j = e^{i\alpha_j}$. 

        \vspace{.05in}

    This time, the monodromy after grafting can be computed as follows. Denoting $E_i(w)$ to be the elliptic elements as before, and $T$ to be the map $z\mapsto \lambda z$, we have from   \cite[Lemma 5.5]{Dum} that the monodromy after bending will be conjugate to $E_1(a_1)\circ E_2(a_2)\circ\dots\circ E_r(a_r)\circ T$, which is computed to be $z\mapsto \lambda\omega_1^{-1}\omega_2^{-1}\dots\omega_r^{-1} z+c$, for $c=(a_1-1)+\sum \omega_1^{-1}\omega_2^{-1}\dots\omega_i^{-1}(a_{i+1}-a_i)$ (where $a_{r+1}=\lambda$). Since $\lambda>1$, the monodromy is either a loxodromic or hyperbolic element.

        \vspace{.05in}

     As before, we now compute the Schwarzian derivative of the developing map, and split into two cases depending on the total angle $\alpha$ of the leaves of the lamination spiralling onto the geodesic boundary component:

        \vspace{.05in} 

    \begin{enumerate}
        \item \textbf{$\alpha = 0$ :} (see also Lemma 3.4. of \cite{Gup}) In this case, recall that the weight on the geodesic boundary is infinite, and hence we perform an ``infinite grafting" on any lift in the universal cover.  This amounts to attaching a ``logarithmic end", or ``semi-infinite lune" $\mathcal{L}_\infty$, i.e.\  semi-infinite chain of $\cp$-s each slit along an identical arc, to any such lift. (See \S4.2 of \cite{Gup-Mj}.) Take such a lift to be the vertical geodesic $\textrm{Re}(z)=0$ in $\mathbb{H}^2$, with the lift of the surface lies to its right. The infinite grafting does not change the monodromy of the end, so it remains $\langle z \rightarrow \lambda z \rangle$.  Recall that the semi-infinite lune that is grafted in, descends to what is conformally a punctured disk $\mathbb{D}^\ast$ on the surface (\textit{c.f.} the proof of Lemma 4.3 in \cite{Gup-Mj}). The developing map restricted to the universal cover of the punctured disk is thus the conformal diffeomorphism $f:H \to \mathcal{L}_\infty$  defined by $z \mapsto ie^{\textrm{log}(\lambda) z}$. (Recall here that we are taking $H/\langle z \mapsto z +1 \rangle \cong \mathbb{D}^\ast$ as before,) 
        
        
        This time, the  developing map descends to the map   $g(w) =  iw^{\frac{1}{2\pi i}\textrm{log}(\lambda)}$ on $\mathbb{D}^\ast$, which has a Schwarzian derivative given by $S(g) = \frac{4\pi^2+\log(\lambda)^2}{8\pi^2w^2}dw^2$, which has clearly a pole of order 2.

           \vspace{.05in}

        \item \textbf{$\alpha\neq 0$ :} In this case, we first identify a suitable fundamental domain for the universal cover $B$ of the geodesic boundary end for which it is easier to compute a uniformizing map. As earlier, consider the domain $U := \{ z \in \mathbb{H}^2: 1\leq\textrm{Re}(z) < \lambda\}$, and let $V$ be its image under the map $z\mapsto \frac{1}{z}$. Then, $V$ is the region  bounded by the semicircles joining $0$ to $\frac{1}{\lambda}$ and $0$ to $1$, in the lower half-plane (the shaded region in Figure ~\ref{fig2}). The grafting geodesics thus become semicircles in the lower half-plane joining $0$ to $\frac{1}{a_i}$. We now follow the same argument as before. 

        \vspace{.05in} 
        
        First, note that the monodromy around the puncture is a loxodromic element that fixes $0$, given by $z\mapsto \lambda^{-1} e^{i\alpha}z$ (this is using the monodromy formula we found earlier, conjugated by the map $z\mapsto \frac{1}{z}$). Next, grafting introduces lunes of weight $\alpha_i$ along the semicircles in the lower-half plane. This results in a total angle of $\alpha$ between the tangents at $0$ to the boundary semicircles after grafting.  Using these observations, this time  the developing map descends to the map $g(w) =  w^{\frac{1}{2\pi}(\alpha+i\textrm{log}(\lambda))}$ on  $\mathbb{D}^\ast$. Indeed, working in the punctured disk, the two images of a radial slit under $g$  have an angle of $\alpha$ between them, and the monodromy element  mapping one to the other is given by multiplication with $(e^{2\pi i})^{\frac{1}{2\pi}(\alpha+i\log(\lambda))}=\lambda^{-1}e^{i\alpha}$ as desired.

        Computing the Schwarzian derivative, we get $S(g) = \frac{4\pi^2-(\alpha+i\log(\lambda))^2}{8\pi^2w^2}dw^2$. Since $\alpha+i\log(\lambda) \neq \pm 2\pi$, we again get a pole of order $2$.  
    \end{enumerate}  

\begin{figure}
\centering
\def\svgwidth{0.5\columnwidth}
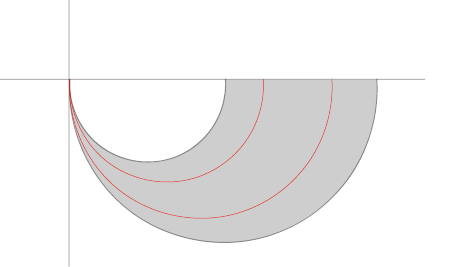
    \centering
    \caption{Grafting a geodesic boundary end: a fundamental domain in the universal cover (shown shaded) with the lifts of the spiralling weighted geodesic (shown in red). }
    \label{fig2}
\end{figure}

        \vspace{.05in}

    Now, it only remains to consider the case of geodesics spiralling clockwise into the end, i.e. the geodesics are semicircles in the upper half-plane. Note that if we take the map $z\mapsto\frac{1}{z}$, $B$ goes to its complex conjugate in $\mathbb{C}$ while the geodesics go to the lines $\textrm{Re}(z)=\frac{1}{a_i}$ in the lower half-plane. (See Figure 2.)  This is clearly the conjugate image of the earlier case of anticlockwise spiralling geodesics, with the grafting geodesics being given by $\textrm{Re}(z)=\frac{\lambda}{a_i}$ in the domain $U$ defined earlier. Thus, it follows that if $z\mapsto f(z)$ is the developing map for the earlier case, the new developing map is its Schwarz reflection, $z \mapsto \overline{f(\bar{z})}$. From this, it follows that the holonomy will be conjugate to $z\mapsto\lambda\omega_1\omega_2\dots\omega_r z +c$, for $c = (\frac{\lambda}{a_r}-1)+ \sum \omega_r\omega_{r-1}\dots\omega_{r+1-i} (\frac{\lambda}{a_{r-i}}-\frac{\lambda}{a_{r+1-i}}).$ This time,  developing map descends to a map $g$  on the punctured disk of the form $w\mapsto w^{\frac{1}{2\pi}(\alpha-i\textrm{log}(\lambda))}$, with Schwarzian derivative $S(g) = \frac{4\pi^2-(\alpha-i\log(\lambda))^2}{8\pi^2w^2}dw^2$  having a pole of order $2$.
    
\end{itemize}   \end{proof}

\noindent
\textit{Remark.} As a result of our computation in the above proof, we obtain a description of the grafting configuration that gives rise to poles of order $1$ - it occurs if and only if we graft a cusp end with $\alpha=2\pi$ and $c\neq0$. Thus, we have:

\begin{coro}\label{cor:simp} 
    If there is a pole of order $1$, it must be obtained by grafting a cusp end along a measured lamination with total weight of leaves going into the cusp equal to $2\pi$. Conversely, generically we obtain a pole of order $1$ by grafting at a cusp when the total weight of the leaves of the lamination going into the cusp is $2\pi$ -- the only exception is when $c=0$ (see \eqref{cdef}).
\end{coro}

\noindent \textbf{Example 1:} Consider the projective structure on $\mathbb{CP}^1\backslash \{0,1,\infty\}$, i.e.\ the thrice punctured sphere, obtained by the trivial chart via the inclusion of this surface in $\mathbb{CP}^1$. Then, if one does the inverse grafting construction via constructing a pleated plane, one would obtain that the pleated plane would be the ideal triangle in the interior of the ball $\mathbb{H}^3$ with vertices at $\{0,1,\infty\} \in \partial_\infty \mathbb{H}^3$. So, a grafting description of this structure would be obtained by taking a hyperbolic sphere with three cusps and grafting in lunes of weight $\pi$ along three geodesic lines running  between pairs of cusps (see Figure 3).  Since the developing  map is just the identity map, its Schwarzian derivative  is identically zero. Now, we can compute the constant $c$ (see \eqref{cdef}) in this case by considering the geodesics going into the cusp $\mathbb{H}/\langle z\mapsto z+1 \rangle$ given by $a_1 = 0, a_2 = \frac{1}{2}$ (according to the notation used in Lemma ~\ref{regular_grafting}) both having weight $\pi$. The computation yields $c = 0 + (-1)(\frac{1}{2}-0) + (1)(1-\frac{1}{2}) = 0$. This illustrates the final statement of Corollary ~\ref{cor:simp} -- even though the total weight at each puncture is $2\pi$, they are not poles of order $1$.

\vspace{0.1 in}

\noindent \textbf{Example 2:} Consider another projective structure on $\mathbb{CP}^1\backslash \{0,1,\infty\}$, obtained by the developing map that descends to the map $f(z)=\textrm{log}(z)+\frac{1}{z}$ (More precisely, this structure descends from a developing map on the intermediate cover $\mathbb{C}\backslash\{2n\pi i : n \in \mathbb{Z}\}$ given by $z \mapsto z + e^{-z}$). As in the previous case, one can perform the inverse grafting operation by constructing a pleated plane, to obtain the following grafting description for this structure: take the thrice punctured sphere and graft along two weighted geodesics: one having both ends going into $1$ and of weight $\pi$, and the other with ends going into $1$ and $0$, with a weight of $2\pi$. (See Figure 3.) Note that since the punctures corresponding to $\infty$ and $1$ have total weight of lamination not equal to $2\pi$, and they are poles of order $2$. Computing the constant $c$ at the puncture corresponding to $0$, we get $c = 1  \neq 0$. So, we expect $0$ to be a pole of order $1$. We can compute the Schwarzian derivative of the map $f$ at the punctures to get that indeed $1$ and $\infty$ are poles of order $2$, while $0$ is a simple pole. This verifies Corollary ~\ref{cor:simp} in this case.

\vspace{.1in} 

\begin{figure}[h]
\centering
\def\svgwidth{0.55\columnwidth}
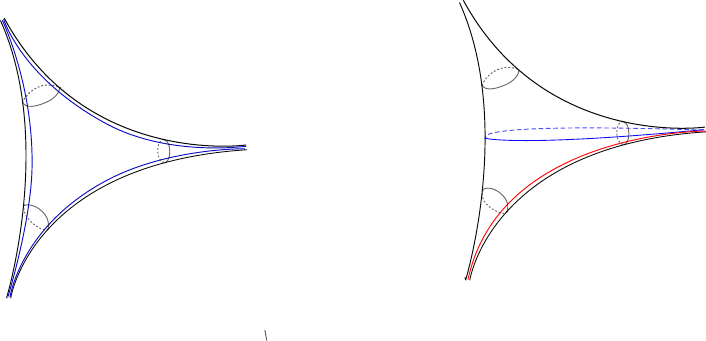
    \centering
    \caption{Grafting description of the projective structures in Example 1 (left) and Example 2 (right). The blue geodesics have weight $\pi$ and the red geodesic has weight $2\pi$.}
    \label{fig1}
\end{figure}

    We also note the following observation culled from the proof of the above Lemma (see also Lemma 3.2 of \cite{Gup}):

    \begin{coro}\label{cor:mon} 
    The following information can be inferred from the monodromy around a regular singularity obtained by grafting a cusp or geodesic boundary end:
    
    \begin{itemize}

    \item[(i)]  The type of end (i.e. cusp or geodesic boundary) and the case of the latter, the length of the geodesic boundary.

    \item[(ii)]  The total weight of the leaves of the grafting lamination incident at that end,  up to positive integer multiples of $2\pi$.
    
    \end{itemize}

    \end{coro}

    \begin{proof}
     We simply note the following from the proof of Lemma ~\ref{irreg_graft}:
    \begin{itemize}
        \item[(i)]  If the monodromy is not loxodromic, we can infer the type of end to be a cusp. If the monodromy is loxodromic and conjugate to the map $z\mapsto \lambda z$, then the length of the boundary is given by $|\log(\lambda)|$.
        \item[(ii)] If the monodromy is parabolic or identity, clearly the total weight is $0$ modulo $2\pi$. Otherwise, the monodromy is conjugate to $z\mapsto \lambda z$, and the total weight of leaves is given by $\arg(\lambda)$ modulo $2\pi$. 
    \end{itemize}  
    \vspace{-0.5cm}
    \end{proof}

 Since the  Lemmas \ref{irreg_graft} and \ref{regular_grafting} involve local computations in a neighborhood of each pole, by combining them we obtain:
 \begin{prop}
 The result of grafting a marked hyperbolic surface $X\in \mathcal{T}(\mathbb{S},\mathbb{M})$ along a measured lamination $(L,\mu)\in \mathcal{ML}(X)$ is a marked projective structure in $\mathcal{P}(\mathbb{S},\mathbb{M})$.
 \end{prop}\label{grafting_structures}

Finally, we note a Lemma that relates the grafting description at a geodesic or cusp end, to the corresponding exponent of the Schwarzian derivative of the resulting projective structure:
\begin{lem}\label{lem:regexp}
    Suppose we obtain a projective structure by grafting a geodesic boundary end of length $l$ (which can be zero, giving a cusp end) and total weight of the leaves of the grafting lamination spiralling onto that end being $\alpha$. Let $q(z)dz^2$ denote the Schwarzian derivative of the projective structure at the puncture corresponding to the end. Let the exponent of this quadratic differential be $r$. Then, if the geodesics spiral anticlockwise, we have $r=\pm (l-i\alpha)$. If the geodesics spiral clockwise, we have $r=\pm (l+i\alpha)$. 
\end{lem}
\begin{proof}
    This immediately follows from a simple calculation of the exponent using the Schwarzian derivatives we computed in the proof of Lemma ~\ref{regular_grafting}. Recall that the exponent $r$ is given as $r = \pm 2\pi i \sqrt{1-2a}$, where $a$ is the coefficient of $z^{-2}dz^2$ in the Schwarzian derivative. We have the following cases:
    \begin{itemize}
        \item $l=0$, i.e. we have a cusp end:
        \begin{enumerate}
            \item {$\alpha$ is not an integer multiple of $2\pi$:} the exponent is $\pm i\alpha$.
            \item {$\alpha = 2\pi n$ for $n\geq0$ :} 
            in both the cases considered, the coefficient of $z^{-2}dz^2$ in the Schwarzian derivative is $\frac{1-n^2}{2}$, so the exponent is $\pm 2\pi i n = \pm i \alpha$.
        \end{enumerate}
        \item We have a geodesic boundary end with $l>0$ :
        \begin{enumerate}
            \item $\alpha = 0$ : the exponent is $\pm \log(\lambda) = \pm l$.
            \item $\alpha \neq 0$ : If the spiralling is anticlockwise, the exponent is $\pm (i \alpha - \log (\lambda))= \pm (l-i\alpha)$. If the spiralling is clockwise, it is $\pm (l+i\alpha)$.
        \end{enumerate}
    \end{itemize}
    \vspace{-0.5cm}
    
\end{proof}


\end{subsection}

\subsection{Defining the Unsigned Grafting Map}
In this and the following sections, we omit the pair $(\mathbb{S},\mathbb{M})$ for the parameter spaces if there is no source of confusion.
We define the unsigned grafting map, 
\begin{equation}\label{eq:unsign}
 Gr': \mathcal{TML} \to  \mathcal{P}(\sm)
\end{equation}
where the domain is the space of pairs
$$\mathcal{TML} =  \{ (X,\lambda) : X \in \mathcal{T}, \lambda \in \mathcal{ML}(X)\}$$
which is a quotient of the product of signed spaces $\mathcal{T}^{\pm}(\sm)\times\mathcal{ML}^{\pm}(\sm)$, obtained by identifying the pairs  differing only in the signings. 

Note that there is a forgetful projection map 
 \begin{equation}
\label{map1} F: \mathcal{T}^{\pm}\times\mathcal{ML}^{\pm} \rightarrow \mathcal{TML}
 \end{equation}
that is described as follows:
Let us take an element $(X_{\sigma},\lambda_\tau) \in \mathcal{T}^{\pm}\times\mathcal{ML}^{\pm}$, where $\sigma,\tau$ denote the respective signings. We need to define the corresponding element on the right-hand side. We take $X$ to be the underlying hyperbolic structure of $X_{\sigma}$. To define the measured lamination $\lambda_{\tau}$  it remains to  prescribe the direction in which the leaves of the lamination spiral on entering the geodesic boundary ends. We define the spiralling direction to be clockwise (with respect to the orientation of $X$) if the signs $\sigma$ and $\tau$ match at the end, and anticlockwise otherwise.

\vspace{.05in}

 \noindent \textit{Remark.} There is a natural continuous map $\mathcal{TML}\rightarrow\mathcal{T}$ that takes the pair $(X,\lambda)$ to $X$. However, this map is not a fiber bundle, since the fiber over an element $X$ is $\mathcal{ML}(X)$ which by Proposition \ref{dim_measured_laminations} is homeomorphic to $\mathbb{R}^{\chi-p}\times\mathbb{R}_{\geq0}^{p}$ (where $p$ is the number of cusp ends of $X$), and these fibers are not homeomorphic to each other. In particular, $\mathcal{TML}$ cannot be written as a natural product $\mathcal{T}\times\mathcal{ML}$.

\vspace{.05in}

\noindent We observe:

\begin{prop}\label{prop:cont}
    The map $Gr'$ is continuous.
\end{prop}
\begin{proof}
We only provide a sketch of the proof here since this is essentially identical to the proof of continuity of the grafting map for closed surfaces (which in turn follows from the discussion of the continuity of the       ``quakebend cocycle" in \cite[Chapter II.3.11]{EpMar}.)  

\vspace{.05in} 

In the closed case, the crucial observation is that if two pairs  $(X_1,\lambda_1)$ and $(X_2,\lambda_2)$ are close in $\mathcal{T}_g \times \mathcal{ML}$, the corresponding lifted laminations would be close in the universal cover $\mathbb{H}^2$, suitably normalized by fixing three ideal points. We describe this further: consider a $(1+\epsilon)$-quasi-isometry from $\mathbb{H}^2$ to $\mathbb{H}^2$ (fixing, say, $0,1,\infty \in \partial_\infty \mathbb{H}^2$) that descends to a $(1+\epsilon)$-quasi-isometry between the two surfaces. (Here $\epsilon>0$ is small if the surfaces $X_1$ and $X_2$ are close.) Such an almost-isometry extends to a homeomorphism of the boundary that is close to the identity map. Now recall a measured lamination can be thought of as a Borel measure on $\partial_\infty \mathbb{H}^2 \times \partial_\infty \mathbb{H}^2 \setminus \Delta$ (where $\Delta$ is the diagonal subspace). The fact that the ideal boundary correspondence is close to the identity map then implies that the lifts of same lamination $\lambda \in \mathcal{M})$ on the two hyperbolic surfaces will be close (in the Hausdorff metric) on any compact subset $K$ of the universal cover $\mathbb{H}^2$, where the respective universal covers are now identified via the $(1+\epsilon)$-quasi-isometry mentioned above. (In this argument ``close" means $\epsilon^\prime$-close where $\epsilon^\prime \to 0$ as $\epsilon\to 0$.)  Moreover, 
the topology on $\mathcal{ML}$ is the weak topology on the space of such measures, and nearby laminations in this topology will also be Hausdorff-close on $K$ (see, for example, Proposition 1.9 of \cite{Gabai} that follows arguments in \cite{Pen}.) 

\vspace{.05in} 

As a consequence, if one restricts to such a compact subset $K\subset \mathbb{H}^2$, the resulting developing maps after grafting $(X_1,\lambda_1)$ and  $(X_2,\lambda_2)$ will be close in $\cp$. (Recall that these developing maps are obtained by grafting in ``lunes" along the leaves of the lamination $\lambda_1$ and $\lambda_2$ respectively.) 

\vspace{.05in} 

The only two new cases to consider in our non-closed setting are:
\vspace{.05in} 

{(A)} Grafting along a hyperbolic surface with geodesic boundary, where the grafting lamination spirals on to the boundary. The leaves spiralling on to the boundary component are isolated geodesics, and in the universal cover, lifts to a sequence of geodesic lines accumulating to the lift of the boundary component. (See Figure 3.)  Hence any compact set $K \subset \mathbb{H}^2$ intersects finitely many such lines. As in the closed-surface argument above, as one varies the hyperbolic surface (this might vary the length of the geodesic boundary component), and the lamination (i.e.\ varying the weight of the spiralling leaves),  the picture of the laminations restricted  to $K$ varies continuously, and hence the grafting map is continuous. 

\begin{figure}
  \centering
  \includegraphics[scale=0.34]{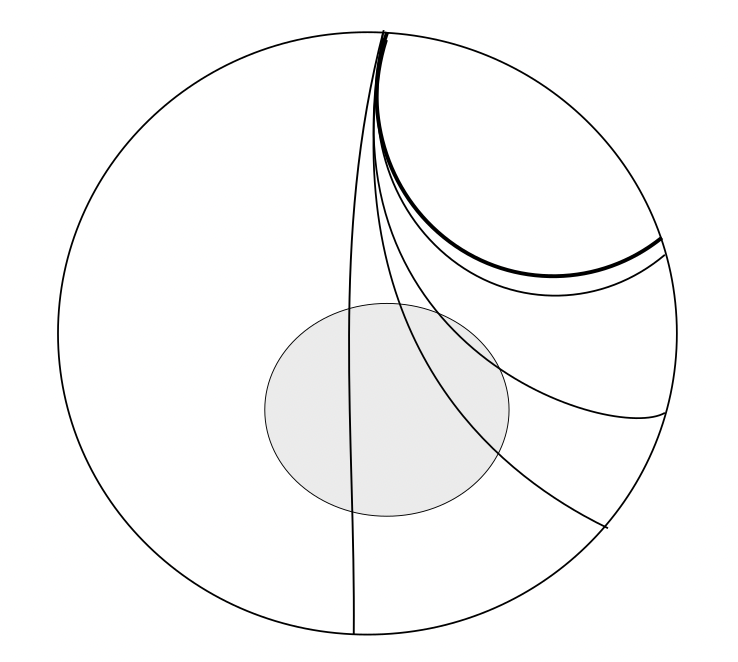}
  \caption{Lifts of a spiralling leaf to the universal cover $\mathbb{H}^2$ accumulates to a lift of the geodesic boundary component. The compact set $K$ (shown shaded) intersects finitely many of such lifts.}
\end{figure}

\vspace{.05in} 

(B) Grafting along a hyperbolic surface with a cusp, with finitely many isolated leaves of the grafting lamination going out of the cusp.  Once again, in the universal cover any compact set $K$ will intersect finitely many of lifts of such leaves. This time, varying the hyperbolic structure includes the possibility of the cusp ``opening out" to a geodesic boundary component (of some small length). It is a consequence of the Collar Lemma that a sequence of hyperbolic surfaces with the length of a geodesic boundary tends to zero converges to a cusped hyperbolic surface in the Gromov-Hausdorff sense (see, for example, Lemma 2.15 of \cite{HarmonicTropical}). In particular, even in this case the restriction of the lamination to a fixed compact subset $K$ will vary continuously, and so will developing map after grafting.

Thus, in both cases, the grafting map is still continuous. \end{proof}



\begin{subsection}{Bijectivity of the Unsigned Grafting Map}
Before we turn to the signed grafting map, we shall first prove that the \textit{unsigned} grafting map defined in \S3.2 is a bijection.

\vspace{.05in} 

First, let us note a Lemma we will use; this is implicitly used in \cite{Gup-Mj} (see the proof of Proposition 4.5 there). In what follows, we refer to Theorem \ref{thm: kulpink} and the brief discussion that follows it. 

\begin{lem}\label{pleating_line}
Let $f : \widetilde{X} \rightarrow \mathbb{CP}^1$ be the developing map of a projective structure on a hyperbolic surface $X$ with a cusp. Let the lift of the cusp to the ideal boundary of $\widetilde{X}$ be denoted $p$. Suppose that there is a point $q\in \mathbb{CP}^1$ such that for $x \in \widetilde{X}$ that continuously varies along a non-constant path in the horocyclic neighbourhood of $p$, there is an embedded disk $U_x \ni x$ in $\overline{\mathbb{H}}$ such that $p \in \partial U_x$, $f(U_x)$ is a round disk in $\mathbb{CP}^1$ tangent to $q$, and $f$ extends continuously over $U_x$ to $p$ such that $f(p)=q$. Then, if the boundary circles of the disks $f(U_x)$ are not tangent to each other at $q$, the pleated plane corresponding to $f$ has an isolated pleating line incident at $q$.
\end{lem}
\begin{proof}
Note that each $U_x$ is contained in a maximal disk $V_x$ on $\widetilde{X}$, since its image is a round disk on $\mathbb{CP}^1$. However, we have that the $V_x$ must be distinct, since if $U_x,U_y \subseteq V$, then $f(V)$ is also a round disk in $\mathbb{CP}^1$, containing the round disks $f(U_x),f(U_y)$, but their boundaries must intersect at the common point $q$, contradicting that the boundaries of $f(U_x),f(U_y)$ are not tangent to each other. Hence, we get a family of maximal disks whose images have $q$ as a common boundary point. Now, consider the interior of the convex hull associated to each maximal disk $V_x$, i.e$.$ the convex hull in $\mathbb{H}^3$ of the set $f(\overline{U_x}\cap \partial_{\infty} \widetilde{X}) \subset \partial_\infty \mathbb{H}^3$. Following \S 4 of \cite{Kul-Pin}, since $V_x$ are distinct maximal balls, it follows that $C(f(\overline{U_x}\cap \partial_{\infty} \widetilde{X}))$ are distinct as well. These define the pleated plane corresponding to the projective structure; in the  totally-geodesic copy of $ \mathbb{H}^2$ obtained by a ``straightening" of this pleated plane, each convex hull is either a plaque or a pleat with one boundary point at $q$. Now, if we have at least two among these family of convex hulls which straighten to a plaque, the plaques must have a nonzero angle between them, which implies that we must have some weight of geodesic lamination incident at $q$. Otherwise, we would have a continuous family of convex hulls such that all of them straighten to pleating lines incident on $q$, which again implies that $q$ will have an incident geodesic pleating lamination of nonzero weight. By Lemma ~\ref{lem:lam}, the geodesic lamination must in fact have a pleating line incident on $q$. 
\end{proof}

Next, we proceed to construct an inverse to the (unsigned) grafting map:

\begin{prop}\label{gr_inv}
Given a projective structure $P\in \mathcal{P}(\mathbb{S},\mathbb{M})$, there is a unique hyperbolic surface $X\in \mathcal{T}(\mathbb{S},\mathbb{M})$ and a unique measured lamination $(L,\mu)\in \mathcal{ML}(X)$ such that $P$ is obtained by grafting $X$ along $(L,\mu)$.
\end{prop}
\begin{proof}

As a consequence of Theorem \ref{thm: kulpink}, we know that $P$ can be obtained by grafting a unique hyperbolic surface homeomorphic to $S_{g,k}$ along a unique measured lamination on it. Now, to prove the proposition, it suffices to show that the grafting lamination and the grafting ends are as described in \S $2.1.$, i.e. a crown end, a cusp end or a geodesic end, with a finite number of geodesics going into the ends. Hence, this is a completely local argument, and henceforth let us assume that $P$ has no irregular singularities. The case for irregular singularities was established in Proposition $4.5.$ of \cite{Gup-Mj}, and our proof for that case follows from their result.

\vspace{.05in}

 Let the pair of hyperbolic surface and measured lamination be $(X',L)$. Then, it follows that $X'$ is of genus $g$ and has $k$ ends, which are either cusps or flares. In the case of a cusp, basic hyperbolic geometry implies that any leaf of $L$ incident entering a horocyclic neighborhood of the cusp must be isolated and exiting out of the cusp end (\textit{c.f.} Lemma \ref{lem:lam}). So, it suffices to show that if $X'$ has a flaring end bounded by a closed geodesic $\gamma$, then there  is a leaf of $\widetilde{L}$ in the universal cover that is a geodesic line incident to one of the end-points of the lift of $\gamma$. This would imply that the leaves of $L$ incident at that end either spirals and  accumulates onto $\gamma$, or is $\gamma$ itself (with infinite weight). In other words, we need to only rule out the possibility that there is a leaf of $L$ that exits the flaring end ``transverse" to $\gamma$.  

We shall do this in the remainder of the proof; for simplicity we shall assume that $X'$ has only one flaring end and no cusps.

\vspace{.05in}


We shall now use the fact that in a neighbourhood of a regular puncture that is conformally $\mathbb{D}^\ast$, the developing map for the projective structure descends to explicit maps expressed in a local coordinate $w$ on $\mathbb{D}^\ast$, obtained by solving the Schwarzian equation \eqref{schw}. (See \S4.1.1 for a discussion.)  In what follows, the neighborhood of the puncture is $\mathbb{D}^\ast = {H}/\langle z \mapsto z+1\rangle$ where $H = \{ z \vert \ \text{Im}(z) > h_0 \}$ is a horodisk centered at $\infty$.

We have the following cases:

\begin{itemize}
    \item $f(w)=w^{-n}+\textrm{log}(w)$ (on the punctured unit disk) or $f(z)=e^{2\pi i\alpha z}$ (on the universal cover $\mathbb{H}^2$) with $\alpha$ real : In this case, the monodromy around the puncture is identity or elliptic or parabolic. We shall also assume that the asymptotic value the puncture is $0\in \cp$. Now, if $\theta \neq 0$, for each $z$ in the horodisk $H$, there is a neighbourhood of $z$ which maps to a round disk in $\mathbb{CP}^1$, centered at $f(z)$ and tangent to $0 \in \mathbb{CP}^1$. Varying $z$ along a horocycle, the images of the disks have continuously varying tangents at $0$. Hence, by Lemma ~\ref{pleating_line}, we get that there is a pleating line incident at $0$, which belongs to a  $\rho_{\mathbb{C}}$-equivariant family of geodesics. On straightening the plane, we get a $\rho_{\mathbb{R}}$-equivariant family of geodesics on the disk. (Recall here that $\rho_\mathbb{C}$ is the original representation into $\pslc$, and $\rho_\mathbb{R}$ is the representation into $\pslr$ obtained after straightening, \textit{c.f.} Theorem \ref{thm: kulpink}.) Since the elliptic or parabolic monodromy around the puncture fixes the point $0$,  we have that the $\rho_{\mathbb{C}}$-equivariant family of pleating lines is also incident at $0$, hence after straightening as well they will be incident at $0$. Thus, the $\rho_{\mathbb{R}}$-image of the peripheral loop around the puncture  will be a hyperbolic or parabolic element. In case of being a hyperbolic element, the geodesics of the lamination will either approach the geodesic fixed by the hyperbolic element, or it will be that fixed geodesic with infinite weight. In either of these cases, we know from the computations in Lemma 4.2. that the monodromy after grafting will be hyperbolic or loxodromic, hence this case is not possible. Thus, it follows that the cusp monodromy will be parabolic, so we obtain that $X'$ has a cusp end with a geodesic lamination going into it.   

    \vspace{.05in}
    
    \item $f(w)=w^\alpha$ (on the punctured disk) with $\alpha$ not real: In this case, instead of varying $w$ along a  horocycle in the universal cover, we vary it along a path $\{(t,\theta(t)):t\in (0,\epsilon)\}$ in polar coordinates on $\mathbb{D}^\ast$ where $\theta(t)= \frac{\alpha_r}{\alpha_i}\textrm{log}(t) + k$ (here $\alpha_r,\alpha_i$ are the real and imaginary parts of $\alpha$ respectively). Then, note that the image of this path under the developing map will be $(te^{i\theta})^\alpha = e^{(\log{t}+i\theta)(\alpha_r+i\alpha_i)}= e^{(\log(t)\alpha_r-\theta\alpha_i) + i(\theta \alpha_r +\log(t)\alpha_i)} = e^{k\alpha + i \frac{|\alpha|^2}{\alpha_i}\log(t)}$, and as $t\rightarrow 0$, $\log(t) \rightarrow -\infty$. Therefore, the image of the path is clearly a circle on $\mathbb{CP}^1$ separating the fixed points of the monodromy. So, it is again easy to construct neighbourhoods of $z$ which map to round disks in $\mathbb{CP}^1$ passing through $0\in \cp$ satisfying the hypotheses of Lemma ~\ref{pleating_line}. 
    Note that since the monodromy around the puncture is loxodromic, we know that the $\rho_\mathbb{R}$-monodromy is hyperbolic. By Lemma ~\ref{pleating_line} the grafting lamination $L$ has a leaf converging to a fixed point of the loxodromic element that is the monodromy around the puncture.
    
\end{itemize}

Hence in all the cases, it follows that on a neighbourhood of a regular singularity, the projective structure is obtained by grafting along a cusp end with some weight of the geodesic lamination going into the cusp (in the case when the monodromy is parabolic, elliptic, or identity), or by grafting a geodesic boundary component having infinite weight  (in the case when the monodromy is hyperbolic) or grafting along weighted  geodesics spiralling onto a geodesic boundary component (in the case when the monodromy is loxodromic).
\end{proof}

\noindent
\textit{Remark.} Note that in order to apply Lemma ~\ref{pleating_line} in the first case above, we only require $f$ to extend continuously over a disc $U_x$ and not over the whole neighbourhood of the lift of a cusp end. This is an important distinction, as we shall see in \S ~\ref{subsubsec: reg_asymp} when we discuss the asymptotic behavior of  the map $z \mapsto z^{-n} + \textrm{log}(z)$.

\subsection{The Signed Grafting Map and proof of Theorem \ref{thm1} }
We shall now define the signed grafting map
 \begin{equation*}   \widehat{\text{Gr}}:\mathcal{T}^\pm(\sm) \times \mathcal{ML}^\pm(\sm)  \to \mathcal{P}^\pm(\sm)
    \end{equation*}
and then prove that it is a homeomorphism (proving Theorem \ref{thm1}). 

\vspace{.1in} 
First, given an interior marked point $p \in \mathbb{P}$, we define a complex function $c_p$ on $\mathcal{T}^{\pm}\times\mathcal{ML}^{\pm}$ as:
\begin{equation} \label{eq:cp}
c_p (X_\sigma, \mu_\tau) = l +i\alpha
\end{equation} 

where $\alpha$ and $l$ are the signed weight of geodesic lamination and signed length of geodesic boundary at $p$.

Now define the map
\begin{equation} \label{cmap}
c : \mathcal{T}^{\pm}\times\mathcal{ML}^{\pm} \rightarrow \mathbb{C}^m 
\end{equation} 
as \[
(X_\sigma,\mu_\tau) \mapsto (c_{p_1},c_{p_2},...,c_{p_m})
\]
where $p_1,p_2,...,p_m$ are the interior marked points $\mathbb{P}$ and $c_p$ is the complex parameter defined above.

\vspace{.1in}

We already have defined an unsigned grafting map $Gr^\prime : \mathcal{TML} \to \mathcal{P}(\sm)$. Composing with the projection $F: \mathcal{T}^{\pm}\times\mathcal{ML}^{\pm} \to \mathcal{TML}$, we get a map $Gr'\circ F : \mathcal{T}^{\pm}\times\mathcal{ML}^{\pm} \to \mathcal{P}$. This map is not injective: more precisely, suppose $(X,\lambda) \in \mathcal{T}^{\pm}\times\mathcal{ML}^{\pm}$ such that there is a  geodesic boundary component of $X$ with a leaf of $\lambda$ spiralling onto it, then in the signed spaces there is a sign associated with boundary, as well as the leaf, and by our conventions, when both these signs agree, then they spiral in exactly the same way. The map $Gr'\circ F$ would then take both these signed pairs to the same projective structure.  
We need to then argue that in such a case, there is nevertheless a map to the signed space of projective structures $\mathcal{P}^{\pm}$ that is injective. This is most efficiently described using the map $c$ defined above, and the notion of the fiber product (see Definition \ref{def:fiberprod}), as we shall now see. 

\vspace{.1in}

Recall that the space of signed meromorphic structures $\mathcal{P}^{\pm}$ introduced in \S ~\ref{signed_proj_struct} is a fiber product $\mathcal{P}(\mathbb{S},\mathbb{M})\times_{\mathbb{C}^{m}}\mathbb{C}^{m}$ where $\lvert \mathbb{P} \rvert = m$, with respect to the map $r^2:  \mathcal{P}(\mathbb{S},\mathbb{M}) \to \mathbb{C}^m$ and squaring map $ \text{sq}: \mathbb{C}^m \to \mathbb{C}^m$.  Here, recall that $r$ records the exponents (defined up to sign) of the Schwarzian derivative at the punctures in $\mathbb{P}$.

\vspace{.1in} 

By Lemma ~\ref{fiber_product}, in order to define a map from $\mathcal{T}^{\pm}\times\mathcal{ML}^{\pm}$ to $\mathcal{P}^{\pm}$ that ``lifts" the map $Gr^\prime\circ F$, it suffices to define a map to $\mathbb{C}^m$  and check that its composition with $sq$ equals the composition $r^2\circ( Gr'\circ F)$.   This is exactly the map $c$ defined above; indeed, Lemma \ref{lem:regexp} ensures that we have $sq\circ c = r^2 \circ (Gr'\circ F)$. By the universal property of fiber products (\textit{c.f.} Definition \ref{def:fiberprod}), this defines the signed grafting map $\widehat{Gr}$ (as in the statement of Theorem \ref{thm1}). Note that this also shows that $\widehat{Gr}$ is continuous.
 
\vspace{.2in}

We can now formally complete the proof of our main Grafting Theorem:

\begin{proof}[Proof of Theorem \ref{thm1}]
Recall that we have defined the grafting map $\widehat{Gr}$ by first defining the unsigned grafting map  $Gr^\prime$ (see \eqref{eq:unsign}) and then using the universal property of fiber products (see the preceding discussion). By Proposition \ref{gr_inv}, the map $Gr^\prime$ is a bijection. 
We will show that so is $\widehat{\text{Gr}}$, by working locally at the punctures -- since signings are local and  independent parameters at the punctures, it is sufficient to prove the bijection working locally.   At a puncture $p$, given an unsigned projective structure, there are two signed projective structures with different signings at $p$ that descend to the given structure, except when the exponent at $p$ is zero. Similarly, if we look at the pair $(|l_p|,|\alpha_p|)$  at $p$, there are exactly two signed pairs with different signings at $p$ that descend to it -- either $(\pm l_p, \pm \alpha_p)$ or $(\pm l_p, \mp \alpha_p)$ depending on the direction of spiralling of the incident lamination geodesics -- except when $l_p=\alpha_p = 0$, and  indeed, by Lemma \ref{lem:regexp} the exponent at a regular singularity is $0$ if and only if we graft along a cusp with no leaves of the grafting lamination incident on it. Thus, due to the way the signings are defined, we get that $\widehat{Gr}$ is bijective at the level of fibers over the signed spaces, hence $\widehat{Gr}$ is bijective.

 Also, note that the domain of $\widehat{Gr}$ is a real dimensional cell of dimension $2\chi$, by Theorem ~\ref{dim_teichmuller} and Lemma ~\ref{dim_measured}. Thus, by the bijectivity of $\widehat{Gr}$ and the Invariance of Domain, it follows that $\widehat{Gr}$ is a homeomorphism.
\end{proof}

\end{subsection}

\section{The Monodromy Map}

In \S 6 of \cite{All-Bri}, they define the framed monodromy of a signed projective structure without apparent singularities. The first goal of this section is to extend their definition to the collection of \textit{all} signed projective structures, that is, we want to define a map,
\[
\widehat{\Phi} : \mathcal{P}^{\pm}(\mathbb{S},\mathbb{M})\rightarrow \widehat{\chi}(\mathbb{S},\mathbb{M})
\]
such that $\widehat{\Phi}$ agrees with the map defined in \cite{All-Bri} on the set of signed projective structures without apparent singularities. Then, we shall characterize the image of $\widehat{\Phi}$, proving Theorem \ref{thm2} and finally show that it is a local homeomorphism (Theorem \ref{thm3}). 

\subsection{Constructing the map $\widehat{\Phi}$}
Given a signed and marked meromorphic projective structure $P \in \mathcal{P}^\pm(\sm)$, we shall construct a framing by considering the asymptotic values of the developing map at  the lifts of the marked points $\mathbb{M}$. Recall the following classical notion: 

\begin{defn}[Asymptotic value] \label{defn:asymp}
    An asymptotic value of a (meromorphic) function at an ideal point $p$ is a point in $\cp$ that is a limiting value along a path that diverges to $p$.
\end{defn}

At an irregular singularity, there are exactly $(n-2)$ asymptotic values (see \cite[Corollary 4.1]{Gup-Mj_2}); these can be assigned to the marked points on the corresponding boundary of $\mathbb{S}$.  and the corresponding equivariant assignment of lifts at the universal cover defines the framing at the lifts of such marked points.  Thus, it remains to define the framing at regular singularities, namely at the marked points $\mathbb{P} \subset \mathbb{M}$. First, we determine the asymptotic values at these points.

\subsubsection{Asymptotics of the developing map at regular singularities}\label{subsubsec: reg_asymp}
Let $P$ be defined by a meromorphic quadratic differential $q$, and consider a regular singularity of $P$, which recall is an interior puncture in $\mathbb{S}$. 
We endow the neighbourhood of the puncture with the end-extension topology, as defined in \S 3.1. of \cite{Ball}. Now, we study the asymptotics of these developing maps at the puncture, according to the value of the exponent $r$ of the Schwarzian derivative (see Definition \ref{defn:exp}). 
\vspace{.05in} 

The possible developing maps are obtained by studying the Schwarzian equation \eqref{schw} -- see the discussion in \S2.3 of \cite{Gup} and the references therein.  There is also a discussion regarding these asymptotics in Lemma $4.1.1.$ of \cite{Ball}. 
\vspace{.05in} 
The asymptotics are as follows:


\begin{enumerate}
    \item $r =2\pi i \theta \textrm{ for } \theta \in \mathbb{C}\setminus \mathbb{R} :$ The developing map is of the form $y(z)=z^\theta$  - the monodromy around the puncture in this case is given as multiplication with $e^{2i\pi \theta}$, which is either a hyperbolic or loxodromic element. There are two asymptotic values, $0$ and $\infty$. A pair of paths that limit to these asymptotic values are exactly paths spiralling into the cusp in opposite directions. As a consequence, it is not possible to continuously extend the developing map to the puncture.
    \item $r =2\pi i \theta \textrm{ for } \theta \in \mathbb{R}\backslash \mathbb{Z} :$ The developing map is $y(z)=z^\theta$ and the monodromy around the puncture is elliptic, given by multiplication with $e^{2i\pi  \theta}$. It is possible to continuously extend the developing map to the puncture by setting the value at the puncture to be $0$. As a consequence, it has a unique asymptotic value.
    \item $r =2\pi i n \textrm{ for } n \in \mathbb{Z} :$ The developing map is either $y(z)= z^n$ - in which case the monodromy is identity, or $y(z)=z^{-|n|}+\textrm{log}(z)$ - in this case the monodromy is a parabolic element. In the former case, it is possible to extend the map continuously to the puncture by again setting the value at the puncture to zero. In the latter case, it is not possible to extend the developing map continuously to the puncture. However, there is a unique asymptotic value, $\infty$. To see this, we can write the developing map in polar coordinates as $f(se^{i\theta})=s^{-n}e^{-in\theta}+i\theta + \textrm{log}(s)=(s^{-n}\textrm{cos}(n\theta)+\textrm{log}(s))+i(s^{-n}\textrm{sin}(n\theta)+\theta)$. Now, for $s$ sufficiently small, $|s^{-n}|>>|\textrm{log}(s)|$. Hence, taking $\theta$ to be zero and $s\rightarrow 0$, $f(s)\rightarrow \infty$. However, for $s$ sufficiently small, we can also choose $\theta$ to make the real part of the above zero, add an arbitrary multiple of $2\pi/n$ to $\theta$ so that the magnitude of the imaginary part becomes smaller than $2\pi /n$. Thus, $f$ cannot be extended continuously to $\infty$. To see that the asymptotic value of $\infty$ is unique, note that if $(s(t),\theta(t))$ is a curve converging to the puncture,  we have $s(t)\rightarrow 0$ as $t\rightarrow \infty$. Now if $f(s(t),\theta(t))$ approaches a bounded asymptotic value $\alpha+i\beta$, then clearly $\theta(t)\neq 0$ for $t$ sufficiently large, since otherwise, the real part of $f$, i.e. $s(t)^{-n}+\textrm{log}(s(t))$ would become arbitrarily large. So, for $t$ sufficiently large, $2\pi k/n <\theta(t) < 2\pi (k+1)/n$ for an integer $k$, and moreover $\textrm{cos}(\theta(t))\rightarrow0$. However, then the imaginary part of $f$ will be unbounded, a contradiction.
\end{enumerate}

\noindent \textit{Remark.} Note that in each case the asymptotic values of the developing map determined above are also fixed points of the monodromy around the puncture. Moreover, in the case that the monodromy is loxodromic or parabolic, the fixed points are precisely the asymptotic values - see cases (1) and (3) above. However, when the monodromy is elliptic as in case (2), there would be two fixed points, but the asymptotic value is only one of them. A key advantage with using asymptotic values is that it is also uniquely defined at an apparent singularity, when the peripheral monodromy is the identity element.  

\subsubsection{Defining the framing}

We define the framing at a puncture in $\mathbb{P}$ as follows:

\begin{enumerate}
    \item If the monodromy is identity or parabolic, we define the unique asymptotic value of the developing map as the framing. These asymptotic values are computed in \S4.1.1 for the model developing maps described there, which assume that the puncture is at $0\in\cp$. More generally, the developing map differs from the model map by post-composing with an element of $\pslc$, and the asymptotic values differ accordingly. 
    \item If the monodromy is elliptic or loxodromic, we define the framing by considering the sign of the projective structure at the puncture, exactly as in \cite[\S6]{All-Bri}. (There is such a sign since the exponent around such a puncture is necessarily non-zero; indeed, if $r=0$ then the peripheral monodromy is either identity or parabolic.)  Namely, we assign the framing to be one of the two fixed points of the monodromy around the puncture according to the sign. To be precise, since the fixed points in $\cp$ are determined by the eigenlines corresponding to the two eigenvalues $e^{\pm\lambda}$, we choose the eigenvalue with the same sign in the exponent. By the preceding remark, in the case that the peripheral monodromy is loxodromic, these fixed points are exactly the asymptotic values of the developing map at the puncture. However, in the case where the peripheral monodromy is elliptic, the framing is not necessarily equal to the unique asymptotic value.
\end{enumerate}

Clearly, the association of points in $\mathbb{CP}^1$ with the punctures prescribed gives a well-defined framing $\beta$. Together with the usual monodromy representation $\rho:\pi_1(\mathbb{S} \setminus \mathbb{P}) \to \pslc$, we obtain a framed representation $\hat{\rho} = (\rho,\beta) \in \widehat{\bigchi}(\sm)$.  Recall that we had started with a signed projective structure $P$ at the beginning of the section; the assignment $\widehat{\Phi}(P) = \hat{\rho}$ thus defines the framed monodromy map $\widehat{\Phi}$.  By the preceding remark, our definition agrees with that of Allegretti-Bridgeland in the case where there are no apparent singularities. 

\subsection{Non-degenerate framed representations and flips}


First, we recall the definition of a non-degenerate framed representation
(see \S 4.2 of \cite{All-Bri}, Definition 2.6 of \cite{Gup-Mj} and Definition 2.4 of \cite{Gup}).

\begin{defn}\label{def:deg}
    A framed representation $\hat{\rho} = (\rho, \beta)$ is \textit{degenerate} if one of the following holds:
\begin{enumerate}
    \item The image of $\beta$ is a single point $\{p\}$ and the monodromy around each puncture is parabolic with fixed point $p$ or the identity.

    \item The image of $\beta$ has two points $\{p,q\}$ and the monodromy around each puncture  fixes both $p$ and $q$.

    \item If $\tilde{p_1},\tilde{p_2} \in F_\infty$ are lifts of  $p_1,p_2 \in \mathbb{M}$ that are successive points on the same boundary component of $\mathbb{S}$, then $\beta(\tilde{p_1}) \neq \beta(\tilde{p_2})$.
    
\end{enumerate}

    The framed representation is said to be non-degenerate if it is not degenerate. 
\end{defn}

\textit{Remark.} The above notion is related to, but distinct from, the notion of a \textit{non-degenerate representation}, which was introduced in \cite{Gup} while dealing with the case when  the marked and bordered surface had no boundary components (i.e.\ $\mathbb{M} = \mathbb{P}$). In that case, condition (3) of Definition \ref{def:deg} holds vacuously, and a framed representation is non-degenerate if the underlying representation is non-degenerate (see \cite[Proposition 3.1]{Gup}). However, the converse is not true in the presence of apparent singularities.  As an example, consider the once-punctured torus, and a representation of its fundamental group given by $\alpha \mapsto \big(\begin{smallmatrix} 1 & 1\\
  0 & 1 \end{smallmatrix}\big), \beta \mapsto \big(\begin{smallmatrix} 1 & -1\\ 0 & 1 \end{smallmatrix}\big)$, where $\alpha,\beta$ are the two generators of the fundamental group. Then, since both $\rho(\alpha),\rho(\beta)$ fix the point $\infty$, and the monodromy around the puncture $[\rho(\alpha),\rho(\beta)] = \big(\begin{smallmatrix} 1 & 0\\
  0 & 1 \end{smallmatrix}\big)$ is the identity element, $\rho$ is a degenerate representation. However, we can frame this representation, by sending the one lift of the puncture to the point $1 \in \mathbb{C}$ and then extending equivariantly. The image of the Farey set under the framing is then the set of integers $\mathbb{Z}\subset \mathbb{C}$, and the framing is non-degenerate by the definition above. 

\vspace{0.1in}

The image under the framing map $\beta$ of  a lift of a puncture is a fixed point for the peripheral monodromy around the puncture. If the monodromy around a puncture is elliptic or loxodromic, we can get a new framing by equivariantly choosing the other fixed point. We define a \textit{flip} of a framing to be a change of framing at a subset of the interior punctures $\mathbb{P}$ of $\mathbb{S}$, which have peripheral monodromy that is either elliptic or loxodromic, by choosing of the other fixed point of that monodromy element. We note the following Lemma, which is in essence Remark 4.4 (v) of \cite{All-Bri} (see also Lemma 9.4 of that paper):

\begin{lem} \label{flip_non-degenerate}
    A framed representation is non-degenerate if and only if all of its flips are non-degenerate.
\end{lem}
    

We note another useful Lemma about framed representations, which is essentially taken from \S 9 of \cite{All-Bri}:
\begin{lem}\label{non-degeneracy}
    Given a framed representation $\hat{\rho} = (\rho, \beta)$, the following are equivalent:
    \begin{enumerate}
        \item The framed representation is non-degenerate.
        \item We can flip the framing so that there exists an ideal triangulation such that the Fock-Goncharov coordinates associated to the triangulation are well-defined.
        \item We can flip the framing so that the image of the framing has at least three points in its image and there is no boundary component with two adjacent marked points having the same framing.
    \end{enumerate}
\end{lem}
\begin{proof}
    $(1)\implies(2)$ is a consequence of Theorem $9.1.$ of \cite{All-Bri}. $(2)\implies (3)$ follows since a ideal triangulation which gives well-defined Fock-Goncharov coordinates must have the property that the set of ideal vertices of a pair of adjacent ideal triangles has at least three distinct points. $(3)\implies (1)$ follows from the Definition \ref{def:deg} of non-degeneracy. 
\end{proof}

\noindent
\textit{Remark.} It follows from the definition of the framing in the previous subsection that changing signs of a signed projective structure at a subset of the punctures  in $\mathbb{P}$ precisely results in the flipping the framing of its framed monodromy at those punctures.

\subsection{Characterizing the image of $\widehat{\Phi}$}

We shall now prove Theorem \ref{thm2} which characterizes the image of the monodromy map $\widehat{\Phi}$. We begin with the following observation:

\begin{lem}\label{lem:immT}
    Let $Z$ be a hyperbolic structure on a connected marked and bordered surface with a nonempty set of marked points $\mathbb{M}$, such that it has at least one geodesic boundary component in the case that $\lvert \mathbb{M} \rvert =1$. Then there exists an immersed ideal triangle $T$ in $Z$ with vertices in $\mathbb{M}$ that lifts to an embedded ideal triangle $\widetilde{T}$ in the universal cover of $Z$. Moreover, if $\mathcal{C}$ is a finite collection of isolated geodesic lines in $Z$ with endpoints in $\mathbb{M}$, then we can choose $T$ such that the interior of $T$ is disjoint from each element of $\mathcal{C}$. 
\end{lem}

\begin{proof}

When $\lvert \mathbb{M}\rvert \geq 3$, there exists an ideal triangle $T$  embedded in $Z$ with ideal vertices at three distinct points in $\mathbb{M}$: connect each pair of such points by arcs such that the arcs are pairwise disjoint, and then take their geodesic representatives. For the second statement, note that we can choose the geodesic sides of $T$ to be either from $\mathcal{C}$, or disjoint from each geodesic line in  $\mathcal{C}$; this would imply in particular  that the interior of $T$  does not intersect any geodesic line in $\mathcal{C}$.

In the case when $\mathbb{M}$ has exactly \textit{two} points (say $p,q$), we consider the ideal triangle $T$ where two of the geodesic sides coincide and is exactly the geodesic between $p$ and $q$, and the third geodesic side is the geodesic arc from one of the points ($p$ or $q$) to itself that goes around the other point.   In the universal cover, such an ideal triangle will lift to an embedded ideal triangle. For the second statement, note that each geodesic in $\mathcal{C}$ starts and ends in the set $\{p,q\}$, and hence we can ensure that the two geodesics we chose to be sides of $T$, either coincide with elements of $\mathcal{C}$ or are disjoint from them.

\begin{figure}[h]
  \centering
  \includegraphics[scale=0.34]{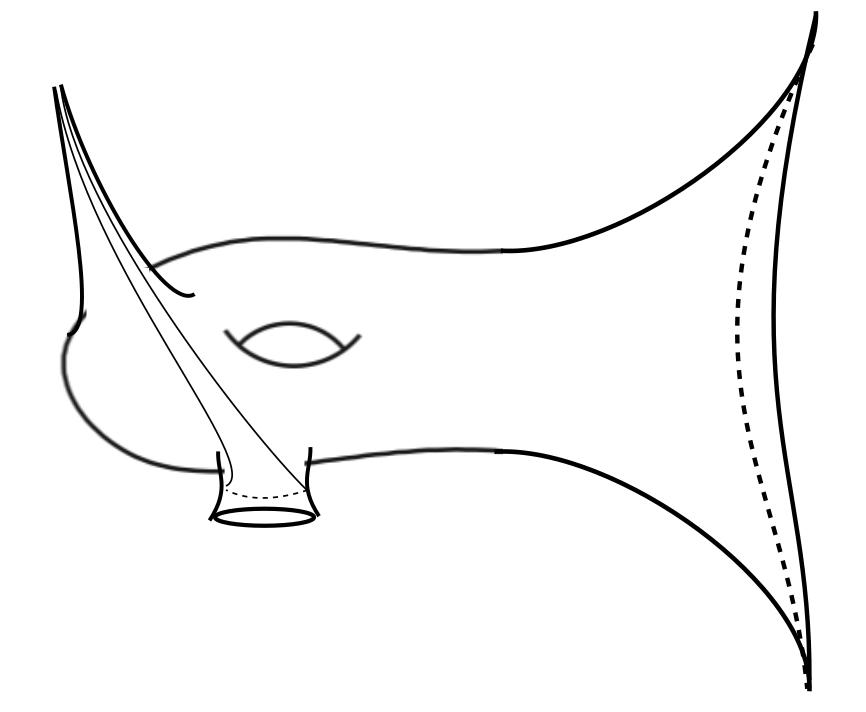}
  \caption{A hyperbolic surface $Z$ as in Lemma \ref{lem:immT}. One can choose three arcs from the marked point $p_0$ (at the cusp end) to itself that wind around the geodesic boundary component and bound an immersed ideal triangle.  }
\end{figure}

Finally, assume that $\mathbb{M}$ has a single point $p_0$.  Recall that in this case we also assume that the hyperbolic surface $Z$ has a geodesic boundary component $c$. (See Figure 5.)  Identify the universal cover $\widetilde{Z}$ as the geodesically convex subset of $\mathbb{H}^2$. The ideal boundary points of $\widetilde{Z}$ will include infinitely many lifts of $p_0$;  this uses the fact that $Z$ has a non-trivial closed geodesic $c$ - lifts of arcs from $p_0$ to itself that twist around $c$ a different number of times will lift to arcs between different lifts of $p_0$. Choose three such lifts of $p_0$, and connect them with geodesic lines; by convexity of $\widetilde{Z}$ this defines an embedded ideal triangle in the universal cover. Once again, we choose the three geodesic lines above to be either from the set $\mathcal{C}$ or disjoint from its elements. The resulting ideal triangle will have its interior disjoint from $\mathcal{C}$. \end{proof}

We shall use the above lemma in the proof of Theorem \ref{thm2}, which recall, states the image of $\widehat{\Phi}$ is precisely the set of non-degenerate framed representations.

\begin{proof}[Proof of Theorem \ref{thm2}]

The proof of one inclusion follows from the constructions in  \cite{Faraco-Gupta} and \cite{Gup-Mj_2}, so we refer the reader to those papers for details. Briefly, let $(\rho,\beta)$ be a given non-degenerate framed representation in $\widehat{\chi}(\mathbb{S},\mathbb{M})$. Then, by Lemma ~\ref{non-degeneracy}, we can flip the framing to $(\rho,\beta')$ and construct a triangulation of the universal cover of the surface $(\mathbb{S}\backslash \mathbb{M})$, such that the Fock-Goncharov co-ordinates for this triangulation are well-defined. Then, we can construct the pleated plane and projective structure from this collection of Fock-Goncharov coordinates (as described in \cite[\S3]{Gup}, \cite[\S3]{Faraco-Gupta}, and \cite[\S3]{Gup-Mj_2}) to obtain a signed projective structure with framed monodromy $(\rho,\beta')$. Finally, we can change signing at punctures in $\mathbb{P}$ to recover the framed monodromy $(\rho,\beta)$, since as noted in \S4.2, changing signs is identical to a flip of the framing. So, in what follows we shall focus on proving the other inclusion. 

\vspace{.05in} 

Let $(P,\mu)$ be a signed meromorphic projective structure in $\mathcal{P}^\pm$, where $\mu$ denotes the choice of a signing on the subset of $\mathbb{P}$ having non-zero exponents. We know by  Proposition ~\ref{gr_inv}, that the underlying unsigned projective structure $P$ is obtained by grafting a pair $(X,\lambda)$, where $X$ is a hyperbolic structure on $(\mathbb{S}, \mathbb{M})$ and $\lambda$ is a measured geodesic lamination. We shall show that the developing image of $P$ has three distinct asymptotic values at the lifts of its marked points at $F_\infty$. Since such an asymptotic value is a fixed point of the monodromy around that puncture, it follows that for some signing $\mu'$, the  corresponding framing will have at least three distinct points in its image, and hence the framed representation of the signed projective structure $(P,\mu')$ is non-degenerate. Since $\mu$ is obtained by changing the sign of $\mu'$, the framed monodromy of $(P,\mu)$ is non-degenerate as well by Lemma ~\ref{non-degeneracy}. Note that the fact that there is no boundary component with adjacent marked points having the same image under the framing follows from the asymptotics of the developing map at irregular singularities, and is discussed in \cite{Gup-Mj}, where it is attributed to \cite[Chapter 8]{Sibuya}. 

\vspace{.05in} 

 The easiest case is when there is an embedded ideal triangle $T$ in $X \setminus \lambda$ whose boundary edges are bi-infinite leaves of the lamination $\lambda$ with vertices in $\mathbb{M}$. This lifts to an ideal triangle $\widetilde{T}$ in the universal cover, with ideal vertices in $F_\infty$. Before grafting the developing images of these three points are distinct; indeed, the developing image of $\widetilde{T}$ embeds in $\cp$. Since the interior of $\widetilde{T}$ is disjoint from the lift $\tilde{\lambda}$ of the bending lamination, its image under the developing map remains embedded after grafting. In particular, the images of the three ideal vertices will be distinct, and are asymptotic values of the developing map of the projective structure $P$. Thus, we conclude that the image of $\beta$ has at least three points, and we are done. 

 Note that the argument works when there is an embedded ideal triangle in the \textit{universal cover} with vertices in $F_\infty$ and that is disjoint from the lift of the bending lamination $\lambda$.

\vspace{.05in} 

An observation that will be useful in what follows is:

\vspace{.05in} 

\textit{Claim.} \textit{An isolated leaf of $\lambda$ is either a simple closed geodesic, or a bi-infinite geodesic between two points in $\mathbb{M}$. (The latter possibility includes the case when both ends of the geodesic spirals onto geodesic boundary components.)} 

\textit{Proof of claim.} This is essentially Corollary 1.7.4 of \cite{Pen} and follows from the structure theory of measured geodesic laminations (see, for example, \cite[Corollary 1.7.3]{Pen}): the only other possibility is to have a leaf spiralling at either (or both) ends onto a compactly supported geodesic lamination, but that is not possible since the finite measure of the accumulating leaf will endow that lamination with infinite transverse measure. $\qed$ 

\vspace{.05in}

\begin{figure}
  \centering
  \includegraphics[scale=0.6]{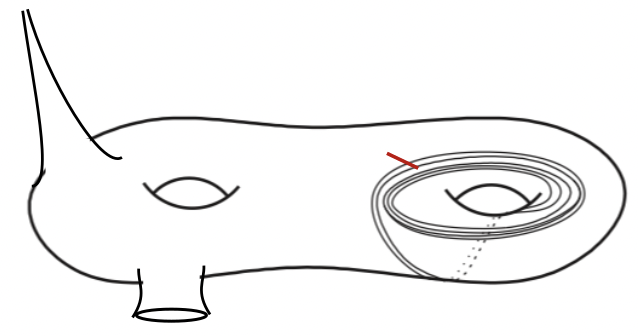}
  \caption{A possible hyperbolic surface $X$ and geodesic lamination $\lambda^\prime$: the metric completion of $X \setminus \lambda^\prime$ is shown in Figure 5.}
\end{figure}

Let $\lambda^\prime \subset \lambda$ be the subset of the lamination supported in a compact part of the surface (i.e.\ away from the ends including the geodesic boundaries). (See Figure 6.) Note that by the claim above, the leaves of $\lambda \setminus \lambda^\prime$ are isolated bi-infinite geodesics between points of $\mathbb{M}$. Let $Y = \widehat{X \setminus \lambda^\prime}$ be the metric completion of the complement of $\lambda^\prime$. (See Figure 5.)  Note that $Y$ is a hyperbolic surface that contains  ends of $X$ corresponding to the points of $\mathbb{M}$, together with some additional ends that are adjacent to $\lambda^\prime$ on $X$, which we shall refer to as the ``lamination-ends". Note that the lamination-ends are either geodesic boundary components or crowns.  

\vspace{.05in}

Let $Z$ be a connected component of $Y$ containing a non-empty set of points from $\mathbb{M}$, and let $\mathcal{C}$ be the collection of isolated leaves of the grafting lamination with endpoints in $\mathbb{M}$ contained in $Z$. We now apply Lemma \ref{lem:immT} to obtain an embedded ideal triangle in the universal cover of $X$, with ideal vertices in $F_\infty$, completely contained in the subset that is the lift of $Z$ (and hence lying in the complement of the lift of $\lambda$). As observed above, the existence of such an embedded ideal triangle suffices to complete the proof.

\vspace{.05in}

The only case where Lemma \ref{lem:immT} will not apply is if $Z$ contains a single point of $\mathbb{M}$, and $Z$ has no geodesic boundary components. In this case, there is a lamination-end of $Z$ which is a crown. Recall that a lamination-end is adjacent to the compactly-supported lamination $\lambda^\prime$ on $X$; the geodesic sides of the crown cannot be isolated leaves of the lamination $\lambda^\prime$, from the claim above. Hence, there are infinitely many (in fact uncountably many) leaves of $\lambda^\prime$ that are accumulating onto any such geodesic side.  (Indeed, from the structure theory of geodesic laminations an arc transverse to $\lambda^\prime$ there will intersect it in a Cantor set.)   In this case consider a homotopically non-trivial arc $\gamma$ from $p_0$ to itself that intersects $\lambda^\prime$, but such that the transverse measure is small. Such a $\gamma$ can be described thus: it starts from $p_0$,  crosses the geodesic side of the lamination-end mentioned above, reaches one of the ``gaps" in the Cantor-set cross-section that also belongs to $Z$, and subsequently remains in $Z$ and returns to $p_0$. In the universal cover, the end-points of its lift $\tilde{\gamma}$ will be points  $p_\pm \in F_\infty$ that are two lifts of $p_0$.  (See Figure 7.)

\begin{figure}
  \centering
  \includegraphics[scale=0.5]{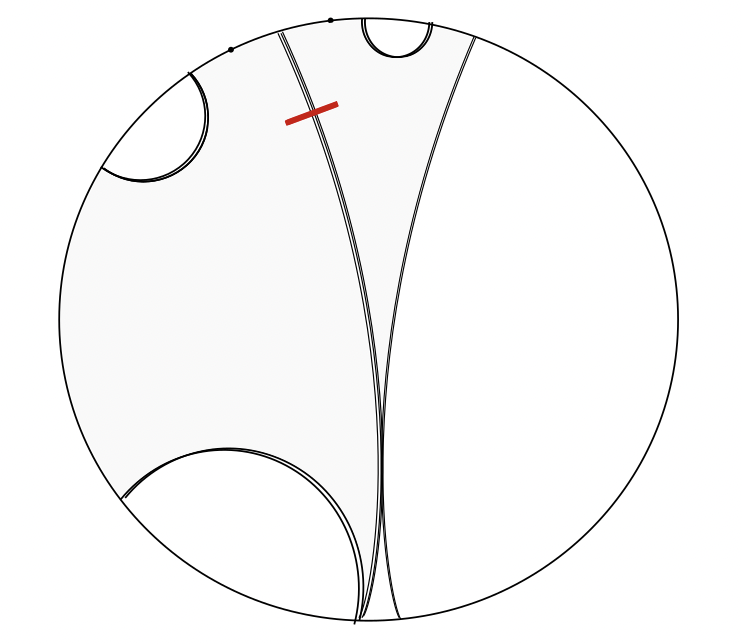}
  \caption{Two complementary regions (plaques) of the lift of the lamination separated by a transverse arc (shown in red) with small measure. If two points in $F_\infty$ lie in the ideal boundary of such plaques, after grafting they will define two distinct asymptotic values of the developing map.}
\end{figure}

Let $\widetilde{Z}_1$ and $\widetilde{Z}_2$ be the corresponding two lifts of $Z$ which $\tilde{\gamma}$ starts and ends in, respectively. Before grafting, these regions of the universal cover develop into disjoint regions in $\cp$; in particular, the developing image of $p_\pm$ are distinct points. This latter fact remains true after grafting, since the transverse measure between $\widetilde{Z}_1$ and $\widetilde{Z}_2$ is small; this implies that for the developing map of $P$, the grafted region inbetween their images has small angular width. Thus we obtain two points in $F_\infty$ whose developing images in $\cp$ are distinct.

 Repeating the argument for another (homotopically distinct) choice of arc from $p_0$ to itself with small intersection with $\lambda^\prime$, that lifts to an arc, say from $p_-$ to another point  $p^\prime\in  F_\infty$, we can conclude that the image of $p_-$ and $p^\prime$ are distinct under the developing map of $P$. Then, a third arc from $p_+$ to $p^\prime$ will also have small transverse measure, since it will be at most the sum of the transverse measures between $p_-$ and $p^\prime$ and between $p_-$ and $p_+$. Once again, before grafting these images of the three points $\{p^\prime, p_{-},  p_+\}$ are distinct, and since after grafting the relative bending between them (which is determined by the transverse measures) is small, they remain distinct. We can then conclude that the three points  have distinct images under the developing map of $P$.  Since the framing $\beta$ for $P$ is determined by the asymptotic values of the developing map, its image has at least three points, and we are done.  \end{proof}

\vspace{.05in} 

We shall now give a characterization of the representations underlying the framed representations in the image of $\widehat{\Phi}$.  Let $$\pi:\widehat{\bigchi}(\sm) \to \bigchi(\mathbb{S} \setminus \mathbb{P})$$ be the forgetful map to the $\pslc$-representation variety of the punctured-surface group $\pi_1(\mathbb{S}\setminus \mathbb{P})$, and let $\Phi = \pi \circ \widehat{\Phi}$ be the un-framed monodromy map. 
As a corollary of Theorem \ref{thm2}, we can characterize the image $\Phi$ as follows (\textit{c.f.} Theorem A of \cite{Faraco-Gupta} for the case when $\mathbb{P} = \mathbb{M}$):
\begin{coro}\label{cor:reps}
Recall that $m$ denotes the number of punctures of $\mathbb{S}$, $k$ is the number of boundary components, and $\{n_i-2\}_{1\leq i\leq k}$ are the numbers of marked points on the boundary components, so that $N= \sum\limits_{i=1}^k(n_i-2)$ denotes the total number of marked boundary points. If $N\geq 3$, then any representation is in the image of $\Phi$, i.e.\ the un-framed monodromy map is surjective. For $N\leq 2$, a representation $\rho$ lies in the image of ${\Phi}$ if and only if one of the following hold:
\begin{itemize}
    \item $\rho$ is a non-degenerate representation. 
    \item $k=0$ and $\rho$ is a degenerate representation with at least one apparent singularity, excluding the following cases: 
    \begin{itemize}
        \item $\rho$ is the trivial representation, for $g>0$ and $m=1$ or $m=2$
        \item the image of $\rho$ is a group of order $2$ and $g>0$, $m=1$
    \end{itemize}
    \item $N=1$ and $\rho$ is a degenerate representation, excluding the following cases:
    \begin{itemize}
        \item $\rho$ is the trivial representation and $m=0$ or $m=1$, 
        \item the image of $\rho$ is a group of order $2$ and $m=0$.
    \end{itemize}
    \item $N=2$ and $\rho$ is a degenerate representation, excluding the case where $\rho$ is the trivial representation and $m=0$.
\end{itemize}
\end{coro}
\begin{proof} By Theorem \ref{thm2}, it suffices to show that these are the only representations that can be framed to obtain a non-degenerate framed representation. This is a consequence of the following four cases/observations:

\vspace{.05in}

(i) If $\rho$ is non-degenerate, we can arbitrarily assign a framing to obtain a non-degenerate $({\rho},\beta)$, by the first part of \cite[Proposition 4.1]{Gup}. If $N \geq3$, then again we can arbitrarily assign a framing to the marked points on boundaries so that the image has at least $3$ distinct points, making it non-degenerate.

\vspace{.05in}

 (ii) In the case $k=0$, we are reduced to the cases considered in Theorem A of \cite{Faraco-Gupta}: If $\rho$ is degenerate and has no apparent singularities, any framed representation will also be degenerate by \cite[Proposition 4.1]{Gup}, which cannot arise as the framed monodromy of a projective structure by \cite[Theorem 6.1]{All-Bri}. If $\rho$ is degenerate and has an apparent singularity, we reduce to the following cases:
    \begin{enumerate}
        \item[(a)] $\rho$ is the trivial representation: In this case, is it easy to check that by Definition \ref{def:deg} the framed representation will be non-degenerate if and only if the image of the framing consists of at least $3$ points. Thus, one can define a non-degenerate framing in this case if and only if $m\geq 3$.

        \vspace{.05in}

        \item[(b)] $\rho$ is not the trivial representation and $m\geq2$ : In this case, suppose that the punctures are $\{p_1,p_2,...,p_m\}$ and $p_1$ is an apparent singularity. We shall use the same labels for their lifts to a fundamental domain in the universal cover of $\mathbb{S}$. Let $Q$ denote the set of points that are fixed by all elements of the image of $\rho$, as in the definition of a degenerate representation. Now, since the image of $p_1$ under the framing can be arbitrary, and $\rho$ is non-trivial, we can  choose a point to be the image of $p_1$ and extend equivariantly such that there are at least two distinct images of the lifts of $p_1$ under the $\rho$-equivariant framing, and moreover these images are disjoint from $Q$. Thus, the image of the resulting framing has at least three points, ensuring that it is non-degenerate.

        \vspace{.05in}

        \item[(c)] $\rho$ is not the trivial representation and $m=1$ : Let the only puncture be $p_1$, which is also an apparent singularity. If the image of $\rho$ is a group of order two, then clearly the image of the framing can have at most two points, and those set of points will be fixed by the image of $\rho$. Thus, it would be a degenerate framed representation by Definition \ref{def:deg}. On the other hand if the image of $\rho$ is a group of order more than two, clearly we can choose a point $p\in \cp$ such that its orbit under the action of this group consists of at least three points. Thus, we can define the framing by mapping a lift of the puncture to $p$ and extending equivariantly; the resulting framed representation is non-degenerate because the image of the framing has at least three points. So, in this case, the representation has such a framing if and only if the image of $\rho$ is not a group of order two.
        
    \end{enumerate}

    \vspace{.05in}

    (iii)  Suppose the total number $N$ of marked points on boundaries be equal to $1$, and $\rho$ is a degenerate representation. If the image of $\rho$ is a group of order at least $3$, we can frame the marked boundary points so that its image under the action of $\rho$ also has at least $3$ points, hence we would obtain a non-degenerate framing. Now we consider the rest of the cases:
    \vspace{.05in}

    \begin{enumerate}
        \item[(a)] $\rho$ is the trivial representation: Clearly if $m\leq1$, the image of the framed representation has exactly $2$ points which are fixed points of the monodromy around all loops, hence the representation is degenerate. If $m\geq2$, then we can arbitrarily assign a framing to the punctures and the marked boundary points to obtain a framing with $3$ points, getting a non-degenerate representation.

        \vspace{.05in}

        \item[(b)] The image of $\rho$ is a group of order $2$: If $m=0$, the image of any framing will necessarily consist of at most $2$ points. Moreover, since the image of $\rho$ is abelian, the monodromy around the boundary must be identity. Hence, it follows that the framed representation is degenerate. If $m\geq1$, then we can define a framing that sends the (lift of the) puncture to a point fixed by the monodromy around all loops, and sends a lift of the marked point on the boundary to a point in $\cp$ whose orbit under the image of $\rho$ has order $2$. Then again, we obtain a framing with $3$ points, making it non-degenerate.
        
    \end{enumerate}

    \vspace{.05in}

    (iv)  Finally, if the total number of marked boundary points is $2$, and there is at least $1$ apparent singularity, we can assign points in $\cp$ to lifts of the marked boundary to a fundamental domain arbitrarily, and extend equivariantly to obtain a framing with at least $3$ points, making it a non-degenerate framing. If there are no punctures, then again we can obtain $3$ points in the image of the framing if $\rho$ is not the trivial representation. If $\rho$ is trivial, then clearly any framed representation will be degenerate.
\end{proof}

\noindent\textit{Remark.} We note that the projective structures in \cite{Faraco-Gupta} and their resulting \textit{framed} monodromies are consistent with the characterization in Theorem \ref{thm2}, i.e.\ they are non-degenerate framed representations. For example, in the case that the number of punctures  $m\geq 2$, and $\rho$ is a non-trivial affine representation (i.e.\ with image in the affine group $\text{Aff}(\mathbb{C})\subset \pslc$), then $\rho$ is a degenerate representation. However, the corresponding affine structure with holonomy $\rho$ constructed in the proof of \cite[Theorem C]{Faraco-Gupta} does define a non-degenerate framing by considering the asymptotic values of the developing map as in \S4.1.2.

\subsection{Showing $\widehat{\Phi}$ is a local homeomorphism}

In this section, we prove that the monodromy map  is a local homeomorphism (Theorem \ref{thm3}).  For the case of projective structures with only irregular singularities, this was proved in \cite{Gup-Mj}, and we shall use results from there, together with our Grafting Theorem (Theorem \ref{thm1}).

\vspace{.1in} 

Before we start with the proof, we record the following observation that we shall use; this is also \cite[Lemma 5.4]{Baba3}.

\begin{lem}\label{lem:isotopy}
    Let $E$ be either a cuspidal end or a collar neighborhood of a geodesic boundary component of a hyperbolic surface, and let $L_1$ and $L_2$ be two distinct collections of pairwise-disjoint weighted geodesics incident at that end, i.e.\ either going into the cusp  or spiralling onto the boundary geodesic. Suppose that (i) the total sum of weights for both collections are the same,  (ii) in case of a geodesic boundary end, the leaves spiral in the same direction, (iii) the resulting monodromy around the puncture is the same in both cases. Then, the developing maps on the universal cover of the end $E$ after grafting along $L_1$ and $L_2$ respectively, are $\mathbb{Z}$-equivariantly isotopic. 
 \end{lem}

 \begin{proof}

Let $\widetilde{E}$ be a lift of $E$ to the universal cover. We shall focus our attention on a fundamental domain $F \subset \widetilde{E}$ of the $\mathbb{Z}$-action. The geodesic lines in $L_1$ and $L_2$ lift to a collection of pairwise-disjoint geodesic lines passing through $F$, each asymptotic to the same point in the ideal boundary. Here, we choose the fundamental domain $F$ such that there are finitely many such lifts in $F$. (See Figure 2 for the case of a geodesic boundary end.) 
Grafting along $L_1$ (and $L_2$) inserts lunes along these lines in the universal cover, of angular widths equal to the corresponding weights. Note that the intermediate regions of $F$ can be thought of as lunes of zero angular width (i.e. regions bounded on two sides by circular arcs incident to a common point in $\cp$ where they share a tangent line). After grafting, the new fundamental domain of the $\mathbb{Z}$-action on the target is a portion of a lune in $\cp$ of angular width equal to the sum of the weights on the grafting lines. Since we are also assuming that the monodromy in both cases is identical, the two circular arcs bounding this fundamental domain can be taken to be exactly the same.  The remaining boundary edge of $F$ might result in distinct arcs after grafting (that depends on the locations of the grafting lines). However these arcs will be isotopic to each other, and this isotopy can be extended to a  $\mathbb{Z}$-equivariant isotopy on $\widetilde{E}$.  \end{proof}

\vspace{.05in}

\noindent
\textit{Remark.} The monodromy around the puncture can differ even if the hypotheses (i) and (ii) in the statement of the Lemma are satisfied, as the following example illustrates: Let $E = H/\langle z \mapsto z+1 \rangle$ be a cusp,  let $L_1$ comprise two geodesics, both with weight $\pi$, which lift to the vertical lines $\text{Re}(z) =0$ and $\text{Re}(z)=1/3$, and let $L_2$  comprise two geodesics with weight $\pi$, this time lifting to vertical lines  $\text{Re}(z) =0$ and  $\text{Re}(z) =\frac{1}{2}$. To compute the monodromy from equations \eqref{mon-exp} and \eqref{cdef}, we use $a_1=0, \omega_1 = \omega_2 = e^{i\pi}=-1$, which implies that $c=1-2a_2$ where $a_2$ is the real coordinate of the second geodesic. Thus, in the first case $a_2=1/3 = c$ and the monodromy is a parabolic element, while in the second case $a_2=\frac{1}{2}, c=0$ and the monodromy is the identity element.

\vspace{.1in}

We can now complete:

\begin{proof}[Proof of Theorem \ref{thm3}]
We shall use the strategy of the proof in \cite{Gup-Mj}.
By the Invariance of Domain it suffices to show that $\widehat{\Phi}$ is locally injective.
\vspace{.05in}

Let $\hat{\rho}$ be a framed representation in $\widehat{\bigchi}(\sm)$ that lies in the image of $\widehat{\Phi}$ and let $P \in \mathcal{P}^\pm(\sm)$ such that $\widehat{\Phi}(P) = \hat{\rho}$. 

\vspace{.05in} 

By Theorem \ref{thm1}, we know that there exists $(X, \lambda) \in \mathcal{T}^\pm \times \mathcal{ML}^\pm$ such that we obtain the projective structure $P$ by grafting $X$ along $\lambda$. We shall show that there is a neighborhood $U$ of $P$ in $\mathcal{P}^\pm(\sm)$ such that if $P^\prime = \widehat{Gr}(X^\prime, \lambda^\prime)$ for some $(X^\prime, \lambda^\prime) \in U$ has the same framed monodromy $\hat{\rho}$, then $X= X^\prime$ and $\lambda = \lambda^\prime$.  

\vspace{.05in} 

By Corollary \ref{cor:mon} we know that the monodromy around any regular puncture determines the type of the corresponding end of $X$ and $X^\prime$ (including, in the case of a  geodesic boundary end, its length) and the total weight (modulo $2\pi$) of the leaves of $\lambda$ that are incident at that end. Since $\lambda$ and $\lambda^\prime$ are close to each other, the homotopy classes of the leaves are identical and the corresponding transverse measures are close; in our context this implies that the leaves incident to the cusp or geodesic boundary end of $X$ (and $X^\prime$)  will have \textit{identical} total weights. Moreover, in the case of a geodesic boundary end with leaves of $\lambda$ of \textit{non-zero} weight spiralling into it, one can ensure that the neighborhood $U$ in $\mathcal{T}^\pm \times \mathcal{ML}^\pm$ is small enough such that the direction of spiralling of the leaves are the same.  Since, as already mentioned, the geometry of such an end $E$ is the same for $X$ and $X^\prime$,  we conclude from Lemma ~\ref{lem:isotopy} that one can modify $P$ and $P^\prime$ by an isotopy (note that this does not change the corresponding points in the deformation space) such that the restriction of their developing maps to the universal cover of $E$ are identical.  

\vspace{.05in} 
Proposition 6.3 of \cite{Gup-Mj} asserts that the same holds at irregular singularity: for each such singularity, the framed representation $\hat{\rho}$ determines the corresponding hyperbolic crowns of $X$ and $X^\prime$, as well as the leaves (and weights) of $\lambda$ and $\lambda^\prime$ that intersect the crown end. 
\vspace{.05in} 

Thus, as in \cite{Gup-Mj}, we are reduced to an application of the Ehresmann-Thurston Principle for manifolds with boundary (see, for example, Theorem I.1.7.1 of \cite{EpMar} or Proposition 1 of \cite{Danciger}). Briefly,  let $S_0$ be the compact surface-with-boundary obtained by removing the ends of the marked-and-bordered surface $\mathbb{S}$, and  consider the space $\mathcal{D}$ of developing maps for projective structures on $S_0$ such that on each boundary component they all restrict to the same map. Then the Ehresmann-Thurston Principle asserts that there is a neighborhood of any developing map $D_0$  in this relative deformation space  $\mathcal{D}$ such that any other developing map in that neighborhood with identical holonomy will be equivariantly isotopic to $D_0$. We apply this to conclude that the projective structures $P$ and $P^\prime$ determine the same point in $\mathcal{P}^\pm(\sm)$. \end{proof}

\subsection{Proof of Corollary \ref{cor:biholo}}

Following the Remark at the end of \S 2.3.3, there are non-constant holomorphic functions $r_p : \mathcal{P}^{\pm}\rightarrow \mathbb{C}$ that map signed projective structures to the exponent at each puncture $p$ in $\mathbb{P}$ (which corresponds to a regular singularity). We know from the computations in the proofs of Lemmas ~\ref{regular_grafting} and ~\ref{lem:regexp} that if a puncture $p$ is an apparent singularity for the projective structure, then the exponent $r_p = 2\pi in$ for some $n\in \mathbb{Z}$. Therefore, it follows that the projective structures having apparent singularities are contained in the analytic subset of $\mathcal{P}^{\pm}$ locally cut out by finitely many equations of the form $\{r_p = 2\pi i n_p : p\in\mathbb{P}\}$ for a tuple of integers $(n_p)_{p\in\mathbb{P}} \in \mathbb{Z}^{\lvert \mathbb{P}\rvert }$. 

\vspace{.05in}

By Theorem 1.1 of \cite{All-Bri}, we know that $\widehat{\Phi}$ is a holomorphic map on the subset $\mathcal{P}^{*} \subset \mathcal{P}^\pm$ of signed projective structures with no apparent singularities. Therefore, by Riemann's removable singularity theorem (see, for example,  Proposition 1.1.7 of \cite{Huy05}), it follows that $\widehat{\Phi}$ is a holomorphic map on all of $\mathcal{P}^{\pm}$. Combining this with Theorem \ref{thm3}, we conclude that the map $\widehat{\Phi}$ is a local biholomorphism.
$\qed$

\end{document}

%% file: grafting_domain.pdf_tex
\begingroup%
  \makeatletter%
  \providecommand\color[2][]{%
    \errmessage{(Inkscape) Color is used for the text in Inkscape, but the package 'color.sty' is not loaded}%
    \renewcommand\color[2][]{}%
  }%
  \providecommand\transparent[1]{%
    \errmessage{(Inkscape) Transparency is used (non-zero) for the text in Inkscape, but the package 'transparent.sty' is not loaded}%
    \renewcommand\transparent[1]{}%
  }%
  \providecommand\rotatebox[2]{#2}%
  \newcommand*\fsize{\dimexpr\f@size pt\relax}%
  \newcommand*\lineheight[1]{\fontsize{\fsize}{#1\fsize}\selectfont}%
  \ifx\svgwidth\undefined%
    \setlength{\unitlength}{224.49751221bp}%
    \ifx\svgscale\undefined%
      \relax%
    \else%
      \setlength{\unitlength}{\unitlength * \real{\svgscale}}%
    \fi%
  \else%
    \setlength{\unitlength}{\svgwidth}%
  \fi%
  \global\let\svgwidth\undefined%
  \global\let\svgscale\undefined%
  \makeatother%
  \begin{picture}(1,0.57014806)%
    \lineheight{1}%
    \setlength\tabcolsep{0pt}%
    \put(0,0){\includegraphics[width=\unitlength,page=1]{grafting_domain.pdf}}%
    \put(0.15419568,0.41662093){\color[rgb]{0,0,0}\transparent{0.97100002}\makebox(0,0)[lt]{\lineheight{1.25}\smash{\begin{tabular}[t]{l}$0$\end{tabular}}}}%
    \put(0.4729402,0.42243754){\color[rgb]{0,0,0}\transparent{0.97100002}\makebox(0,0)[lt]{\lineheight{1.25}\smash{\begin{tabular}[t]{l}$\frac{1}{\lambda}$\end{tabular}}}}%
    \put(0.55552324,0.42358931){\color[rgb]{0,0,0}\transparent{0.97100002}\makebox(0,0)[lt]{\lineheight{1.25}\smash{\begin{tabular}[t]{l}$\frac{1}{a_2}$\end{tabular}}}}%
    \put(0.79719605,0.4202737){\color[rgb]{0,0,0}\transparent{0.97100002}\makebox(0,0)[lt]{\lineheight{1.25}\smash{\begin{tabular}[t]{l}$1$\end{tabular}}}}%
    \put(0.6935441,0.42305437){\color[rgb]{0,0,0}\transparent{0.97100002}\makebox(0,0)[lt]{\lineheight{1.25}\smash{\begin{tabular}[t]{l}$\frac{1}{a_1}$\end{tabular}}}}%
  \end{picture}%
\endgroup%

%% file: examples.pdf_tex
\begingroup%
  \makeatletter%
  \providecommand\color[2][]{%
    \errmessage{(Inkscape) Color is used for the text in Inkscape, but the package 'color.sty' is not loaded}%
    \renewcommand\color[2][]{}%
  }%
  \providecommand\transparent[1]{%
    \errmessage{(Inkscape) Transparency is used (non-zero) for the text in Inkscape, but the package 'transparent.sty' is not loaded}%
    \renewcommand\transparent[1]{}%
  }%
  \providecommand\rotatebox[2]{#2}%
  \newcommand*\fsize{\dimexpr\f@size pt\relax}%
  \newcommand*\lineheight[1]{\fontsize{\fsize}{#1\fsize}\selectfont}%
  \ifx\svgwidth\undefined%
    \setlength{\unitlength}{338.8316922bp}%
    \ifx\svgscale\undefined%
      \relax%
    \else%
      \setlength{\unitlength}{\unitlength * \real{\svgscale}}%
    \fi%
  \else%
    \setlength{\unitlength}{\svgwidth}%
  \fi%
  \global\let\svgwidth\undefined%
  \global\let\svgscale\undefined%
  \makeatother%
  \begin{picture}(1,0.48247931)%
    \lineheight{1}%
    \setlength\tabcolsep{0pt}%
    \put(0,0){\includegraphics[width=\unitlength,page=1]{examples.pdf}}%
  \end{picture}%
\endgroup%